\documentclass[11pt]{article}

\usepackage{mathpazo}
\usepackage{amsfonts}\usepackage{amsmath}
\usepackage{bm}
\usepackage{graphicx}
\usepackage{mathtools}
\usepackage{latexsym}
\usepackage[margin=1.05in]{geometry}
  
\usepackage[T1]{fontenc}
\usepackage{times}
\usepackage{color,graphicx}
\usepackage{array}
\usepackage{enumerate}
\usepackage{amsmath}
\usepackage{amssymb}
\usepackage{amsthm}
\usepackage{bbm} 
\usepackage{pgfplots}
\usepackage{pgf}
\usepackage{tikz}
\usepackage{tikz-3dplot} 
\usetikzlibrary{fadings,patterns,arrows.meta,pgfplots.patchplots}
\tikzfading[name=fade right, left color=transparent!0, right color=transparent!35]
\tikzstyle{fadenode}=[circle,fill=lightgray,path fading=fade right,text=black,minimum width=10pt,draw=black]
\usepgfplotslibrary{patchplots} 
\pgfplotsset{width=9cm,compat=1.5.1}
\usepackage{diagbox}

\usepackage{amsthm}
\newtheorem{theorem}{Theorem}

\newtheorem{lemma}[theorem]{Lemma}
\newtheorem{corollary}[theorem]{Corollary}

\newtheorem{proposition}[theorem]{Proposition}

\usepackage{xcolor}
\usepackage{makeidx}
\PassOptionsToPackage{hyphens}{url}
\usepackage[colorlinks=true,linkcolor=blue,anchorcolor=blue,citecolor=red,urlcolor=magenta]{hyperref}
\usepackage{caption}
\usepackage{subcaption}

\newcommand{\rank}{\operatorname{Trank}}
\newcommand{\Wrank}{\operatorname{Wrank}}
\newcommand{\PP}{\operatorname{PP}}
\newcommand{\OPP}{\operatorname{OPP}}
\newcommand{\ACPP}{\operatorname{ACPP}}
\newcommand{\sgn}{\operatorname{sgn}}

\renewcommand{\det}{\operatorname{det}}
\newcommand{\per}{\operatorname{per}}

\DeclareRobustCommand{\stirling}{\genfrac\{\}{0pt}{}}

\newcommand\range{\mathop{\rm range}\nolimits}

\newcommand{\defeq}{:=}

\def\C{\mathbb{C}}

\def\R{\mathbb{R}} 
\def\F{\mathbb{F}}
\def\detnF{\det^n_{\F}}
\def\pernF{\per^n_{\F}}

\def\e{\mathbf{e}}

\def\0{\mathbf{0}}

\definecolor{green1}{rgb}{0.2,0.95,0.475}
\definecolor{green2}{rgb}{0.2,0.8,0.4}
\definecolor{green3}{rgb}{0.2,0.7,0.35}

\begin{document}

\title{A New Formula for the Determinant and \\ Bounds on Its Tensor and Waring Ranks}  

\author{
	Robin Houston, \ \ \ Adam P.~Goucher, \ \ \ and \ \ \ Nathaniel Johnston\footnote{Department of Mathematics \& Computer Science, Mount Allison University, Sackville, NB, Canada E4L 1E4}
}

\date{May 27, 2024}

\maketitle

\begin{abstract}
    We present a new explicit formula for the determinant that contains superexponentially fewer terms than the usual Leibniz formula. As an immediate corollary of our formula, we show that the tensor rank of the $n \times n$ determinant tensor is no larger than the $n$-th Bell number, which is much smaller than the previously best known upper bounds when $n \geq 4$. Over fields of non-zero characteristic we obtain even tighter upper bounds, and we also slightly improve the known lower bounds. In particular, we show that the $4 \times 4$ determinant over $\F_2$ has tensor rank exactly equal to $12$. Our results also improve upon the best known upper bound for the Waring rank of the determinant when $n \geq 17$, and lead to a new family of axis-aligned polytopes that tile $\R^n$.
\end{abstract}
  

\section{Introduction}\label{sec:intro}

The determinant and permanent of an $n \times n$ matrix $A$ are defined by
\begin{align}\label{eq:det_per_defn}
    \det(A) = \sum_{\sigma \in S_n} \left(\sgn(\sigma)\prod_{i=1}^n a_{i,\sigma(i)}\right) \quad \text{and} \quad \per(A) = \sum_{\sigma \in S_n} \left(\prod_{i=1}^n a_{i,\sigma(i)}\right),
\end{align}
respectively, where $S_n$ is the symmetric group over the set $[n] = \{1,2,\ldots,n\}$. There are numerous other explicit formulas for the permanent of a matrix, such as Ryser's formula \cite{BR91}
\begin{align}\label{eq:perm_ryser}
    \per(A) & = \sum_{S \subseteq[n]}\left(\sgn(S)\prod_{i=1}^n \sum_{j \in S} a_{i,j}\right),
\end{align}
where $\sgn(S) = (-1)^{|S|+n}$, as well as Glynn's formula \cite{Gly10}
\begin{align}\label{eq:perm_glynn}
    \per(A) & = \frac{1}{2^{n-1}}\sum_\delta \left(\sgn(\delta)\prod_{i=1}^n\sum_{j=1}^n\delta_ia_{i,j} \right),
\end{align}
where $\sgn(\delta) = \prod_{k=1}^n \delta_k$ and the outer sum ranges over all vectors $\delta \in \{-1,1\}^n$ with $\delta_1 = 1$.

These alternative formulas for the permanent, despite looking more complicated than the defining formula of Equation~\eqref{eq:det_per_defn}, have a very similar form: they each consist of a sum of products of $n$ factors, with each factor in the product being a linear combination of entries from a single row of $A$. One of the advantages of the formulae in Equations~\eqref{eq:perm_ryser} and~\eqref{eq:perm_glynn} is that there are fewer terms in the outer sum ($2^n - 1$ and $2^{n-1}$, respectively, instead of the $n!$ terms of the defining Equation~\eqref{eq:det_per_defn}), so they can be implemented via fewer multiplications.

On the other hand, while there are numerous known methods of computing the determinant of a matrix, most of them do not provide an explicit formula with fewer terms being summed than the defining formula of Equation~\eqref{eq:det_per_defn}. For example, cofactor expansions are just factored forms of Equation~\eqref{eq:det_per_defn} that still consist of a sum of $n!$ terms, and most numerical methods (e.g., those based on Gaussian elimination or matrix decompositions) are iterative (see \cite{Rot01}, for example) and/or require division by entries of the matrix. The only progress in the direction of finding more efficient explicit formulas for the determinant that we are aware of comes from the fact that, when $n = 3$, there are several ways that are known to write the determinant as a sum of just $5$ terms instead of $3! = 6$ (see \cite{Der15,IT16,KM21}, and the references therein, for example), such as
\begin{align}\begin{split}\label{eq:det_3x3}
    \det(A) & = (a_{1,2} + a_{1,3})(a_{2,1} + a_{2,3})(a_{3,1} + a_{3,2}) \\
    & \quad - a_{1,2}a_{2,1}(a_{3,1} + a_{3,2} + a_{3,3}) \\
    & \quad - a_{1,3}(a_{2,1} + a_{2,2} + a_{2,3})a_{3,1} \\
    & \quad - (a_{1,1} + a_{1,2} + a_{1,3})a_{2,3}a_{3,2} \\
    & \quad + a_{1,1}a_{2,2}a_{3,3}.
\end{split}\end{align}

Derksen noticed that, when combined with cofactor expansions, formulas like Equation~\eqref{eq:det_3x3} generalize to give explicit formulas for the $n \times n$ determinant that consist of a sum of $(5/6)^{\lfloor n/3\rfloor}n!$ terms \cite{Der15}. This formula has the fewest terms in the sum known until now, except over fields of characteristic $2$, where $\det(A) = \per(A)$ (since $-1 = 1$) and Ryser's formula of Equation~\eqref{eq:perm_ryser} is a sum of just $2^n - 1$ terms.\footnote{Glynn's $2^{n-1}$-term formula of Equation~\eqref{eq:perm_glynn} does not apply in this setting, since we cannot divide by $2$ in characteristic $2$.}

Our main contribution is to present a new explicit formula for the determinant (Theorem~\ref{thm:main_formula}) that improves upon both of these bounds. Our formula reduces to exactly the $5$-term formula of Equation~\eqref{eq:det_3x3} when $n = 3$, and in general it consists of a sum of exactly $B_n$ terms, where $B_n$ denotes the $n$-th Bell number (i.e., the number of partitions of $[n]$). Since
\[
    \frac{B_n}{(5/6)^{\lfloor n/3\rfloor}n!} \leq \frac{(4n/5)^n}{(\ln(n+1))^n(5/6)^{\lfloor n/3\rfloor}n!} \leq \frac{n^n}{(\ln(n+1))^n n!} \leq \frac{1}{e}\left(\frac{e}{\ln(n+1)}\right)^n,
\]
for all $n \geq 1$ (the first inequality above uses the bound $B_n \leq (4n/5)^n/(\ln(n+1))^n$ from \cite{BT10}), our formula has superexponentially fewer terms than the previously best-known formula. When working over fields of non-zero characteristic, our formula simplifies even further (Corollary~\ref{cor:tensor_rank_field}), to the point of giving a $(2^n - n)$-term formula when the characteristic is $2$ (Corollary~\ref{cor:tensor_rank_permanent_char2}), narrowly surpassing Ryser's $(2^n - 1$)-term formula.

\subsection{Tensor Rank}\label{sec:tensor}

We can regard the $n \times n$ determinant over a field $\F$ as a tensor living in $(\F^n)^{\otimes n}$, and we can then ask questions of it like we ask of any tensor. In particular, we can ask what its \emph{tensor rank} is \cite{Lan12}. That is, if we use $\detnF \in (\F^n)^{\otimes n}$ to denote the $n \times n$ determinant tensor over the field $\F$, what is the least integer $r$ for which there exist $\{\mathbf{v_{j,k}}\} \subset \F^n$ with
\begin{align}\label{eq:det_tensor_sum}
    \detnF = \sum_{k=1}^r \mathbf{v_{1,k}} \otimes \mathbf{v_{2,k}} \otimes \cdots \otimes \mathbf{v_{n,k}}?
\end{align}
We denote tensor rank (i.e., minimal $r$) by $\rank$, so the tensor rank of the determinant is denoted by $\rank(\detnF)$.

Our interest in the tensor rank comes from the fact that every determinant formula present in this paper corresponds to a tensor decomposition of the form of Equation~\eqref{eq:det_tensor_sum} by replacing each occurrence of $a_{i,j}$ in the formula by $\mathbf{e_j}$ (the $j$-th standard basis vector of $\F^n$) in the $i$-th tensor factor of the tensor decomposition. For example, the defining formula~\eqref{eq:det_per_defn} corresponds to the tensor decomposition
\[
    \detnF = \sum_{\sigma \in S_n} \sgn(\sigma) \mathbf{e}_{\sigma(1)} \otimes \mathbf{e}_{\sigma(2)} \otimes \cdots \otimes \mathbf{e}_{\sigma(n)},
\]
which shows that $\rank(\detnF) \leq |S_n| = n!$. Similarly, the formula for the $3 \times 3$ determinant from Equation~\eqref{eq:det_3x3} immediately gives us the following tensor decomposition of $\det^3_{\F}$, which demonstrates that $\rank(\det^3_{\F}) \leq 5$:
\begin{align*}
    \det^3_{\F} & = (\mathbf{e_2} + \mathbf{e_3}) \otimes (\mathbf{e_1} + \mathbf{e_3}) \otimes (\mathbf{e_1} + \mathbf{e_2}) \\
    & \quad - \mathbf{e_2} \otimes \mathbf{e_1} \otimes (\mathbf{e_1} + \mathbf{e_2} + \mathbf{e_3}) \\
    & \quad - \mathbf{e_3} \otimes (\mathbf{e_1} + \mathbf{e_2} + \mathbf{e_3}) \otimes \mathbf{e_1} \\
    & \quad - (\mathbf{e_1} + \mathbf{e_2} + \mathbf{e_3}) \otimes \mathbf{e_3} \otimes \mathbf{e_2} \\
    & \quad + \mathbf{e_1} \otimes \mathbf{e_2} \otimes \mathbf{e_3}.
\end{align*}

In fact, it is known that $\rank(\det^3_{\F}) = 5$ over every field \cite{KM21}, so this decomposition is optimal. When $n \geq 4$, the exact value of $\rank(\detnF)$ is not known, and until now the best known upper bounds on it were exactly the bounds that we discussed earlier:\footnote{It was mentioned in \cite{KM21} that $\rank(\det^5_{\F}) \leq 20$ and $\rank(\det^7_{\F}) \leq 100$, but these are typographical errors; the authors meant $\rank(\det^4_{\F}) \leq 20$ and $\rank(\det^5_{\F}) \leq 100$, which come from the formula $(5/6)^{\lfloor n/3\rfloor}n!$.}
\begin{align*}
    \rank(\detnF) & \leq (5/6)^{\lfloor n/3\rfloor}n! && \text{for all fields $\F$, and} \\
    \rank(\detnF) & \leq 2^n - 1 && \text{if $\F$ has characteristic $2$.}
\end{align*}
Our formula improves upon these bounds by showing that $\rank(\detnF) \leq B_n$ regardless of the field $\F$ (Corollary~\ref{cor:tensor_rank}), and $\rank(\detnF) \leq 2^n - n$ if $\F$ has characteristic $2$ (Corollary~\ref{cor:tensor_rank_permanent_char2}). These upper bounds are already known to be tight when $n = 3$, and we show that they are also tight when $n = 4$ and the ground field has two elements (i.e., the tensor rank is exactly $2^n - n = 12$ in this case; see Theorem~\ref{thm:exactly_12}). We also obtain some other (tighter than $B_n$) upper bounds when $\F$ has any non-zero characteristic (Corollaries~\ref{cor:tensor_rank_field} and~\ref{cor:tensor_rank_field_simple}).

\subsection{Waring Rank}\label{sec:waring}

Another notion of the rank of the determinant comes from thinking of it as a homogeneous polynomial in the $n^2$ entries of the matrix on which it acts. That is, we can think of the determinant as the following degree-$n$ polynomial in the $n^2$ variables $\{x_{i,j}\}$ (we abuse notation slightly and use $\detnF$ to refer to both this polynomial, as well as the tensor from Equation~\eqref{eq:det_tensor_sum}, but which one we mean will always be clear from context):
\begin{align}\label{eq:det_polynomial}
    \detnF(x_{1,1},x_{1,2},\ldots,x_{n,n}) = \sum_{\sigma \in S_n} \left(\sgn(\sigma)\prod_{i=1}^n x_{i,\sigma(i)}\right).
\end{align}
We are interested in the \emph{Waring rank} of this polynomial. That is, what is the least integer $r$ for which there exist linear forms $\{\ell_k\} \subset \operatorname{Hom}(\mathbb{F}^{n\times n},\mathbb{F})$ and scalars $\{c_k\}\subset \mathbb{F}$ with
\begin{align}\label{eq:det_poly_waring_sum}
    \detnF = \sum_{k=1}^r c_k\ell_k^n?
\end{align}
We denote the Waring rank of this determinant polynomial by $\Wrank(\detnF)$.

For example, if $\mathbb{F}$ has characteristic not equal to $2$ then the Waring rank of the degree-$2$ polynomial $xy$ is $2$, since
\[
    xy = \frac{1}{4}\left((x+y)^2 - (x-y)^2\right),
\]
is a linear combination of $2$ squares of linear forms, and it is not possible to write $xy$ as a linear combination (i.e., scalar multiple) of the square of just a single linear form. Similarly, it is well-known that if the characteristic of $\mathbb{F}$ is $0$ or strictly larger than $n$ then the Waring rank of the $n$-variable polynomial $x_1x_2\cdots x_n$ is exactly $2^{n-1}$ as shown in \cite{RS11}.\footnote{If the characteristic is between $2$ and $n$ (inclusive) then the Waring rank of $x_1x_2\cdots x_n$ is infinite: it cannot be written as a linear combination of \emph{any} number of $n$-th powers of linear forms.} By replacing each term in a formula for the determinant with a linear combination of $2^{n-1}$ different $n$-th powers of linear forms, we immediately get the following (also well-known) simple relationship between the tensor and Waring ranks of the determinant:

\begin{lemma}\label{lem:TrankWrank_bound}
    Let $\mathbb{F}$ be a field with characteristic $p$. If $p = 0$ or $p > n$ then
    \[
        \Wrank(\detnF) \leq 2^{n-1}\cdot\rank(\detnF).
    \]
\end{lemma}

When combined with the bounds on $\rank(\detnF)$ from Section~\ref{sec:tensor}, this lemma tells us that if $\mathbb{F}$ is a field with characteristic $0$ or strictly larger than $n$, then
\begin{align}\label{eq:Wrank_big_bound}
    \Wrank(\detnF) & \leq 2^{n-1}(5/6)^{\lfloor n/3\rfloor}n!.
\end{align}
This upper bound was improved in \cite{JT22}, in the case when $\mathbb{F}$ also contains a primitive root of unity (e.g., if $\mathbb{F} = \mathbb{C}$), to
\begin{align}\label{eq:Wrank_small_bound}
    \Wrank(\detnF) & \leq n \cdot n!.
\end{align}

Our formula improves upon these bounds by showing that $\Wrank(\detnF) \leq 2^{n-1} \cdot B_n$, even without the primitive root of unity assumption (Corollary~\ref{cor:tensor_rank}). Our bound is strictly better (i.e., smaller) than the one provided by Inequality~\eqref{eq:Wrank_big_bound} for all $n \geq 4$, and is better than the one provided by Inequality~\eqref{eq:Wrank_small_bound} for all $n \geq 17$.

\subsection{Arrangement of the Paper}

In Section~\ref{sec:main_formula}, we present our main contribution, which is a new formula for the determinant of a matrix (Theorem~\ref{thm:main_formula}). As an immediate corollary, we obtain our new field-independent upper bounds on the tensor and Waring ranks of the determinant (Corollary~\ref{cor:tensor_rank}).

We present two independent proofs of our formula. First, we present a combinatorial proof in Section~\ref{sec:comb_proof}. This proof has the advantage of being rather mechanical (and thus easy to verify), but the disadvantage of not providing much insight into the formula. Second, we present a geometric proof in Section~\ref{sec:geometric} (this section can be read or skipped quite independently of the rest of the paper). This proof has the advantage of providing some insight into \emph{why} the formula works and how it was actually found, but is less direct: it is established for all matrices $A$ in a small open ball in the space of $n \times n$ real matrices, which is sufficient to prove equality of the determinant polynomial and the polynomial in Theorem~\ref{thm:main_formula}, thereby proving the result in full generality over arbitrary commutative rings. This proof also demonstrates some new axis-aligned polytope tilings of $\mathbb{R}^n$, where the 1-skeleta of the polytopes can be naturally identified with flip graphs for ordered partial partitions.

In Section~\ref{sec:fields_non_zero_char}, we investigate what our formula says about the tensor rank of the determinant over fields with non-zero characteristic, and we obtain tighter upper bounds than the field-independent one (Corollaries~\ref{cor:tensor_rank_field}, \ref{cor:tensor_rank_field_simple}, and~\ref{cor:tensor_rank_permanent_char2}). Finally, in Section~\ref{sec:lower_bound} we (very slightly) improve upon the best known \emph{lower} bound for the tensor rank of the determinant over arbitrary fields (Theorem~\ref{thm:det_rank_lb}), we provide a further improvement for the determinant over finite fields (Theorem~\ref{thm:det_rank_lb_f2}), and we show that the $4 \times 4$ determinant over the field with two elements has tensor rank equal to exactly $12$ (Theorem~\ref{thm:exactly_12}), demonstrating optimality of our formula in this case.

\section{The Formula}\label{sec:main_formula}

Before presenting our formula for the determinant, we first need the concept of a partial partition of $[n]$, which is a set of disjoint subsets (called parts) of $[n]$. If the union of a partial partition is $[n]$ then it is a (non-partial) partition. There is a natural bijective correspondence between partial partitions of $[n]$ with no singleton parts and (non-partial) partitions of $[n]$, which works by erasing all singleton sets from a partition or adding singletons of all members that are missing from a partial partition. For example, when $n = 3$, there are $5$ partial partitions of $[n]$ with no singletons and also $5$ partitions of $[n]$ as follows:
\begin{align}\label{eq:PP3}
    \begin{tabular}{ c c c }
    \underline{Partitions} && \underline{Partial partitions with no singletons} \\ 
    $\{\{1,2,3\}\}$ && $\{\{1,2,3\}\}$ \\
    $\{\{1,2\},\{3\}\}$ && $\{\{1,2\}\}$ \\
    $\{\{1,3\},\{2\}\}$ && $\{\{1,3\}\}$ \\
    $\{\{2,3\},\{1\}\}$ && $\{\{2,3\}\}$ \\
    $\{\{1\},\{2\},\{3\}\}$ && $\{\}$
    \end{tabular}
\end{align}

We denote the set of partial partitions of $[n]$ by $\PP(n)$. Just like (non-partial) partitions of $[n]$ give rise to equivalence relations on $[n]$, partial partitions give rise to \emph{partial} equivalence relations on $[n]$: relations that are symmetric and transitive, but need not be reflexive. We denote the partial equivalence relation induced by the partial partition $P$ by $\underset{P}{\sim}$. In other words, for $i, j \in [n]$, $i \underset{P}{\sim} j$ means that there is a part in $P$ containing both $i$ and $j$. Moreover, for $k \in [n]$, $k \underset{P}{\sim} k$ is not guaranteed, since $k$ might not be in any part of $P$.

With the above preliminaries out of the way, we now present our formula for the determinant, which works over any field:

\begin{theorem}\label{thm:main_formula}
    Let $A$ be an $n \times n$ matrix. Then
    \begin{align}\label{eq:main_formula}
        \det(A) & = \ \smashoperator{\sum_{P \in \PP(n)}} \ \ \sgn(P) |P|! \prod_{i=1}^n \begin{cases}\displaystyle \ \ \sum_{j \underset{P}{\sim} i, j \neq i} a_{i,j} & \textrm{ if $i \underset{P}{\sim} i$;} \\ \displaystyle a_{i,i} + \sum_{j \underset{P}{\sim} j}a_{i,j} & \textrm{ if $i \underset{P}{\not\sim} i$,}\end{cases}
    \end{align}
    where $\displaystyle\sgn(P) = \prod_{S \in P}(-1)^{|S|+1}$.
\end{theorem}

We note that the quantity $\sgn(P)$ is equal to the sign of a permutation with cycle type $\{|S| : S \in P\}$, which seems like a quite natural notion for the ``sign'' of a partial partition (e.g., the partial partition $\{\{1,2,3\},\{4,5\}\}$ has the same sign as the permutation $(1,2,3)(4,5)$). While the sum described by Theorem~\ref{thm:main_formula} is over all of $P \in \PP(n)$, if $P$ contains a singleton part $\{i\}$ then that term in the sum equals $0$ since
\[
    \sum_{j \underset{P}{\sim} i, j \neq i} a_{i,j}
\]
is an empty sum. The formula~\eqref{eq:main_formula} can thus be rewritten as a sum over the $P \in \PP(n)$ with no singleton parts. For example, if $n = 2$ then there are two partial partitions of $[n]$ with no singletons: $P_1 = \{\{1,2\}\}$ and $P_2 = \{\}$. We can compute $\sgn(P_1) = -1$, $\sgn(P_2) = 1$, $|P_1|! = 1! = 1$ and $|P_2|! = 0! = 1$, so Theorem~\ref{thm:main_formula} says that
\[
    \det(A) = -a_{1,2}a_{2,1} + a_{1,1}a_{2,2},
\]
which is of course the same as the usual formula for the determinant from Equation~\eqref{eq:det_per_defn}. When $n = 3$, the $5$ partial partitions with no singletons from Equation~\eqref{eq:PP3} result in exactly the $5$-term formula for the determinant that we saw in Equation~\eqref{eq:det_3x3}. When $n = 4$, Theorem~\ref{thm:main_formula} gives the following $15$-term formula for the determinant (surpassing the prior state of the art formula, which has $20$ terms):
\begin{align}\begin{split}\label{eq:det_n4_formula}
    \det(A) & = a_{1,1}a_{2,2}a_{3,3}a_{4,4} \\
    & \quad - (a_{1,2}+a_{1,3}+a_{1,4})(a_{2,1}+a_{2,3}+a_{2,4})(a_{3,1}+a_{3,2}+a_{3,4})(a_{4,1}+a_{4,2}+a_{4,3}) \\
    & \quad + (a_{1,1}+a_{1,2}+a_{1,3}+a_{1,4})(a_{2,3}+a_{2,4})(a_{3,2}+a_{3,4})(a_{4,2}+a_{4,3}) \\
    & \quad + (a_{1,3}+a_{1,4})(a_{2,1}+a_{2,2}+a_{2,3}+a_{2,4})(a_{3,1}+a_{3,4})(a_{4,1}+a_{4,3}) \\
    & \quad + (a_{1,2}+a_{1,4})(a_{2,1}+a_{2,4})(a_{3,1}+a_{3,2}+a_{3,3}+a_{3,4})(a_{4,1}+a_{4,2}) \\
    & \quad + (a_{1,2}+a_{1,3})(a_{2,1}+a_{2,3})(a_{3,1}+a_{3,2})(a_{4,1}+a_{4,2}+a_{4,3}+a_{4,4}) \\
    & \quad - a_{1,2}a_{2,1}(a_{3,1}+a_{3,2}+a_{3,3})(a_{4,1}+a_{4,2}+a_{4,4}) \\
    & \quad - a_{1,3}(a_{2,1}+a_{2,2}+a_{2,3})a_{3,1}(a_{4,1}+a_{4,3}+a_{4,4}) \\
    & \quad - a_{1,4}(a_{2,1}+a_{2,2}+a_{2,4})(a_{3,1}+a_{3,3}+a_{3,4})a_{4,1} \\
    & \quad - (a_{1,1}+a_{1,2}+a_{1,3})a_{2,3}a_{3,2}(a_{4,2}+a_{4,3}+a_{4,4}) \\
    & \quad - (a_{1,1}+a_{1,2}+a_{1,4})a_{2,4}(a_{3,2}+a_{3,3}+a_{3,4})a_{4,2} \\
    & \quad - (a_{1,1}+a_{1,3}+a_{1,4})(a_{2,2}+a_{2,3}+a_{2,4})a_{3,4}a_{4,3} \\
    & \quad + 2a_{1,2}a_{2,1}a_{3,4}a_{4,3} \\
    & \quad + 2a_{1,3}a_{2,4}a_{3,1}a_{4,2} \\
    & \quad + 2a_{1,4}a_{2,3}a_{3,2}a_{4,1}.
\end{split}\end{align}

In general, since there are $B_n$ partitions of $[n]$, there are also $B_n$ partial partitions of $[n]$ with no singleton parts, and thus $B_n$ (potentially) non-zero terms in the formula~\eqref{eq:main_formula}. This demonstrates part~(a) of the following corollary (part~(b) then follows from Lemma~\ref{lem:TrankWrank_bound}):

\begin{corollary}\label{cor:tensor_rank}
    Let $\F$ be a field. Then
    \begin{enumerate}
        \item[(a)] $\rank(\det_{\F}^n) \leq B_n$, and
        
        \item[(b)] if $\F$ has characteristic $0$ or strictly larger than $n$ then $\Wrank(\det_{\F}^n) \leq 2^{n-1} \cdot B_n$.
    \end{enumerate}
\end{corollary}

\section{Combinatorial Proof}\label{sec:comb_proof}

We now present a combinatorial proof of Theorem~\ref{thm:main_formula}. This proof works by just brute-force showing that the formula~\eqref{eq:main_formula}, when expanded as a linear combination of monomials, gives the exact same quantity as the defining formula~\eqref{eq:det_per_defn}. More precisely, let $f : [n] \rightarrow [n]$ be a function (not necessarily a permutation). Our goal is to show that the coefficient of $a_{1,f(1)}a_{2,f(2)}\cdots a_{n,f(n)}$ is the same in Equation~\eqref{eq:main_formula} as it is in the defining formula for the determinant~\eqref{eq:det_per_defn}.

To this end we say that a partial partition $P \in \PP(n)$ is \emph{algebraically compatible} with $f$ if, for all $i \in [n]$, we have the following two properties:
\begin{itemize}
    \item[($\alpha$)] If $i \underset{P}{\sim} i$ then $f(i) \neq i$ and $f(i) \underset{P}{\sim} i$, and
    
    \item[($\beta$)] If $i \underset{P}{\not\sim} i$ then $f(i) = i$ or $f(i) \underset{P}{\sim} f(i)$.
\end{itemize}

If we let $\ACPP(f)$ denote the set of partial partitions that are algebraically compatible with $f$, then Equation~\eqref{eq:main_formula} says exactly that the coefficient $c_{f}$ of the coefficient of $a_{1,f(1)}a_{2,f(2)}\cdots a_{n,f(n)}$ in an expansion of the determinant is equal to
\begin{align}\label{eq:coeff_formula}
        c_f = \ \ \smashoperator{\sum_{P \in \ACPP(f)}} \ \ \sgn(P)|P|!.
\end{align}
    
\begin{lemma}\label{lem:cf_eq_sgn_f}
    Let $c_{f}$ be the coefficient of $a_{1,f(1)}a_{2,f(2)}\cdots a_{n,f(n)}$ after expanding the polynomial in the right-hand side of Equation~\eqref{eq:main_formula}. Then $c_f = \sgn(f)$ if $f$ is a permutation and $c_f = 0$ otherwise.
\end{lemma}

\begin{proof}
    To prove Lemma~\ref{lem:cf_eq_sgn_f} (and thus Theorem~\ref{thm:main_formula}), we now split into two cases.\medskip
    
    \underline{\textbf{Case 1:} $f$ is not a permutation.}

    As $f$ is not surjective, there exists $i \notin \range(f)$. Let
$j = f(i)$ and observe that $j \neq i$. Take an arbitrary $P \in \ACPP(f)$.
Note that $j \sim j$ and, if $i \sim i$, we have $i \sim j$. If $P'$ is
obtained by removing $i$ from the part of $j$ (if $i \sim i$) or introducing
$i$ into the part of $j$ (otherwise), then $P' \in \ACPP(f)$ and has the
opposite sign to $P$. This defines an involution on $\ACPP(f)$ mapping each
algebraically compatible partial partition to one of opposite sign, so
$c_f = 0$.\medskip
    
    \underline{\textbf{Case 2:} $f$ is a permutation.}
    
    We can write $f$ as a product of disjoint cycles of length at least $2$: $f = \sigma_1 \sigma_2 \cdots \sigma_k$. Each cycle $\sigma = (i_1 \ i_2 \ \ldots \ i_\ell )$ corresponds naturally to a subset $S_{\sigma} := \{i_1,i_2,\ldots,i_\ell\} \subseteq [n]$ (though this correspondence is many-to-one since the order of the entries $\sigma$ matters, whereas it does not matter in $S_{\sigma}$). Similarly, from $f$ we can build the partial partition $P_f := \{S_{\sigma_1}, S_{\sigma_2},\ldots,S_{\sigma_k}\}$.
    
    If $i_1$ and $i_2$ are in the same cycle of $f$ then, for any $P \in \ACPP(f)$ we have $i_1 \underset{P}{\sim} i_2$. It follows that $P \in \ACPP(f)$ if and only if $P = P_f$ or $P$ can be obtained from $P_f$ by unioning together some of its parts. In other words, there exists a (non-partial) partition $K = \{K_1,\ldots,K_m\}$ of $[k]$ such that
    \begin{align}\label{eq:cpp_P_perm}
        P = \left\{ \bigcup_{i \in K_j} S_{\sigma_i} : j \in [m] \right\}.
    \end{align}
    
    If $\stirling{k}{m}$ denotes the $(k,m)$-th Stirling number of the 2nd kind, then there are $\stirling{k}{m}$ partitions of $[k]$ with exactly $m$ parts, so there are $\stirling{k}{m}$ partial partitions $P$ of the form~\eqref{eq:cpp_P_perm} with exactly $m$ parts. Each one has $|P| = m$ and $\sgn(P) = (-1)^{k+m}\sgn(f)$, so Equation~\eqref{eq:coeff_formula} can be written more explicitly as
    \begin{align}\begin{split}\label{eq:cf_expand}
        c_f & = \ \ \smashoperator{\sum_{P \in \ACPP(f)}} \ \ \sgn(P)|P|! \\
        & = \sum_{m = 1}^k (-1)^{k+m}\sgn(f)\stirling{k}{m} m! \\
        & = (-1)^k\sgn(f)\left(\sum_{m = 1}^k \stirling{k}{m}(-1)^{m}m!\right).
    \end{split}\end{align}
    We now plug $x = -1$ into the well-known formula
    \[
        \sum_{m = 1}^k \stirling{k}{m} x(x-1)(x-2)\cdots(x-m+1) = x^k
    \]
    to see that
    \begin{align}\label{eq:stirling_sum_neg1}
        \sum_{m = 1}^k \stirling{k}{m}(-1)^{m}m! = (-1)^k.
    \end{align}
    Substituting Equation~\eqref{eq:stirling_sum_neg1} into the bottom line of Equation~\eqref{eq:cf_expand} shows that $c_f = (-1)^{2k}\sgn(f) = \sgn(f)$, which completes the proof.
\end{proof}

\section{Geometric Interpretation and Proof}\label{sec:geometric}

We now present an alternate proof of Theorem~\ref{thm:main_formula} that perhaps provides a bit more insight into why the formula~\eqref{eq:main_formula} works. We take the base field to be $\mathbb{R}$ throughout this section, but remark that this does not lose any generality: proving Theorem~\ref{thm:main_formula} over $\mathbb{R}$ establishes equality between two polynomials in the ring $\mathbb{Z}[a_{1,1}, a_{1,2}, \dots, a_{n,n}]$, which must therefore also hold over any commutative ring.

Throughout the rest of this section, we use 3D terminology (e.g., ``volume'' and ``parallelepiped'') if the dimension $n$ is unknown or greater than $2$. The $n$ standard basis vectors are denoted $\mathbf{e_1}, \ldots, \mathbf{e_n}$. It is notationally convenient to also define $\mathbf{e_0}$ to be the negated sum of the standard basis vectors, so that $\mathbf{e_0} + \mathbf{e_1} + \cdots + \mathbf{e_n} = 0$.

\subsection{Tilings in General}\label{sec:geo_tilings_general}

We consider the matrix $A$ as defining a lattice in $\mathbb{R}^n$, where the lattice points consist of integer linear combinations of the columns of the matrix. More specifically, this lattice is
\[
    \Lambda_A := \{ A\mathbf{v} : \mathbf{v} \in \mathbb{Z}^n \}.
\]
The usual connection between the determinant and this lattice is the fact that a parallelepiped with side vectors equal to the columns of $A$ tiles $\mathbb{R}^n$ by translates in the lattice $\Lambda_A$ and has signed volume equal to $\det(A)$ (see Figure~\ref{fig:parallelogram_tiling}). We will prove Theorem~\ref{thm:main_formula} by constructing a polytope whose volume is given by the formula in Equation~\eqref{eq:main_formula}, but that also tiles $\mathbb{R}^n$ by translates in the lattice $\Lambda_A$ and thus must also have signed volume equal to $\det(A)$ (see Figure~\ref{fig:our_tiling}).\footnote{This technique has been used in the past to create notched cube tilings of $\mathbb{R}^n$ \cite{Ste90}, for example; our tilings will also be axis-aligned polytopes, but will otherwise be slightly more complicated.}

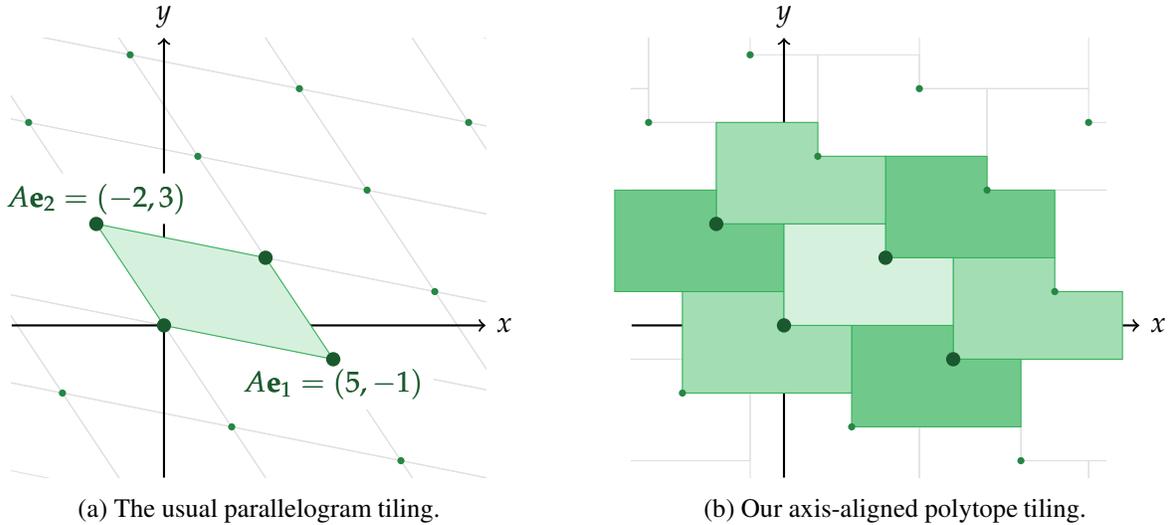
\begin{figure}[!htb]
	\centering
	\begin{subfigure}{.48\textwidth}
		\centering
		\begin{tikzpicture}[scale=0.45]%
    		\coordinate (O) at (0,0);
    		\coordinate (P) at (5,-1);
    		\coordinate (Q) at (-2,3);
    		\coordinate (PQ) at (3,2);
      
            \begin{scope}
                \clip(-4.5,-4.5) rectangle (9.5,8.5);

        		\draw[color=gray!25] (-6,9) -- (-1,8) -- (-3,11) -- (-8,12) -- cycle;
        		\draw[color=gray!25] (-1,8) -- (4,7) -- (2,10) -- (-3,11) -- cycle;
        		\draw[color=gray!25] (4,7) -- (9,6) -- (7,9) -- (2,10) -- cycle;
        		\draw[color=gray!25] (9,6) -- (14,5) -- (12,8) -- (7,9) -- cycle;

        		\draw[color=gray!25] (-9,7) -- (-4,6) -- (-6,9) -- (-11,10) -- cycle;
        		\draw[color=gray!25] (-4,6) -- (1,5) -- (-1,8) -- (-6,9) -- cycle;
        		\draw[color=gray!25] (1,5) -- (6,4) -- (4,7) -- (-1,8) -- cycle;
        		\draw[color=gray!25] (6,4) -- (11,3) -- (9,6) -- (4,7) -- cycle;
        		\draw[color=gray!25] (11,3) -- (16,2) -- (14,5) -- (9,6) -- cycle;
          
        		\draw[color=gray!25] (-7,4) -- (-2,3) -- (-4,6) -- (-9,7) -- cycle;
        		\draw[color=gray!25] (-2,3) -- (3,2) -- (1,5) -- (-4,6) -- cycle;
        		\draw[color=gray!25] (3,2) -- (8,1) -- (6,4) -- (1,5) -- cycle;
        		\draw[color=gray!25] (8,1) -- (13,0) -- (11,3) -- (6,4) -- cycle;
          
        		\draw[color=gray!25] (-5,1) -- (0,0) -- (-2,3) -- (-7,4) -- cycle;
        		\draw[color=gray!25] (0,0) -- (5,-1) -- (3,2) -- (-2,3) -- cycle;
        		\draw[color=gray!25] (5,-1) -- (10,-2) -- (8,1) -- (3,2) -- cycle;
        		\draw[color=gray!25] (10,-2) -- (15,-3) -- (13,0) -- (8,1) -- cycle;
          
        		\draw[color=gray!25] (-8,-1) -- (-3,-2) -- (-5,1) -- (-10,2) -- cycle;
        		\draw[color=gray!25] (-3,-2) -- (2,-3) -- (0,0) -- (-5,1) -- cycle;
        		\draw[color=gray!25] (2,-3) -- (7,-4) -- (5,-1) -- (0,0) -- cycle;
        		\draw[color=gray!25] (7,-4) -- (12,-5) -- (10,-2) -- (5,-1) -- cycle;
          
        		\draw[color=gray!25] (-6,-4) -- (-1,-5) -- (-3,-2) -- (-8,-1) -- cycle;
        		\draw[color=gray!25] (-1,-5) -- (4,-6) -- (2,-3) -- (-3,-2) -- cycle;
        		\draw[color=gray!25] (4,-6) -- (9,-7) -- (7,-4) -- (2,-3) -- cycle;
        		\draw[color=gray!25] (9,-7) -- (14,-8) -- (12,-5) -- (7,-4) -- cycle;
            \end{scope}

    		\draw[thick,-to] (-4.5,0) -- (9.5,0) node[anchor=west]{$x$};
    		\draw[thick,-to] (0,-4.5) -- (0,8.5) node[anchor=south]{$y$};
      
            \filldraw[draw=green3,fill=green3!20!white,fill opacity=0.8] (O) -- (P) -- (PQ) -- (Q) -- cycle;

        	\fill[green3!50!black] (P) circle (1pt) node[anchor=north,fill=white,opacity=0.7,draw=white]{$A\e_1 = (5,-1)$};
        	\fill[green3!50!black] (Q) circle (1pt) node[anchor=south,fill=white,opacity=0.7,draw=white]{$A\e_2 = (-2,3)$};
        	\fill[green3!50!black] (P) circle (6pt) node[anchor=north]{$A\e_1 = (5,-1)$};
        	\fill[green3!50!black] (Q) circle (6pt) node[anchor=south]{$A\e_2 = (-2,3)$};
        	\fill[green3!50!black] (O) circle (6pt);
        	\fill[green3!50!black] (PQ) circle (6pt);

        	\fill[green3!75!black] (-4,6) circle (3pt);
        	\fill[green3!75!black] (-1,8) circle (3pt);
        	\fill[green3!75!black] (-3,-2) circle (3pt);
        	\fill[green3!75!black] (1,5) circle (3pt);
        	\fill[green3!75!black] (2,-3) circle (3pt);
        	\fill[green3!75!black] (4,7) circle (3pt);
        	\fill[green3!75!black] (6,4) circle (3pt);
        	\fill[green3!75!black] (7,-4) circle (3pt);
        	\fill[green3!75!black] (8,1) circle (3pt);
        	\fill[green3!75!black] (9,6) circle (3pt);
    	\end{tikzpicture}
		\caption{The usual parallelogram tiling.}\label{fig:parallelogram_tiling}
	\end{subfigure} \hfill %
	\begin{subfigure}{.48\textwidth}
		\centering
	\begin{tikzpicture}[scale=0.45]%
    		\coordinate (O) at (0,0);
    		\coordinate (P) at (5,-1);
    		\coordinate (Q) at (-2,3);
    		\coordinate (PQ) at (3,2);
      
            \begin{scope}
                \clip(-4.5,-4.5) rectangle (9.5,8.5);

        		\draw[color=gray!25] (-6,-4) -- (-1,-4) -- (-1,-2) -- (-3,-2) -- (-3,-1) -- (-6,-1) -- cycle;
        		\draw[color=gray!25] (-8,-1) -- (-3,-1) -- (-3,1) -- (-5,1) -- (-5,2) -- (-8,2) -- cycle;
        		\draw[color=gray!25] (-10,2) -- (-5,2) -- (-5,4) -- (-7,4) -- (-7,5) -- (-10,5) -- cycle;

        		\draw[color=gray!25] (-1,-5) -- (4,-5) -- (4,-3) -- (2,-3) -- (2,-2) -- (-1,-2) -- cycle;
        		\draw[color=gray!25] (-7,4) -- (-2,4) -- (-2,6) -- (-4,6) -- (-4,7) -- (-7,7) -- cycle;
        		\draw[color=gray!25] (-9,7) -- (-4,7) -- (-4,9) -- (-6,9) -- (-6,10) -- (-9,10) -- cycle;

        		\draw[color=gray!25] (4,-6) -- (9,-6) -- (9,-4) -- (7,-4) -- (7,-3) -- (4,-3) -- cycle;
        		\draw[color=gray!25] (-4,6) -- (1,6) -- (1,8) -- (-1,8) -- (-1,9) -- (-4,9) -- cycle;

        		\draw[color=gray!25] (9,-7) -- (14,-7) -- (14,-5) -- (12,-5) -- (12,-4) -- (9,-4) -- cycle;
        		\draw[color=gray!25] (7,-4) -- (12,-4) -- (12,-2) -- (10,-2) -- (10,-1) -- (7,-1) -- cycle;
        		\draw[color=gray!25] (1,5) -- (6,5) -- (6,7) -- (4,7) -- (4,8) -- (1,8) -- cycle;

        		\draw[color=gray!25] (10,-2) -- (15,-2) -- (15,0) -- (13,0) -- (13,1) -- (10,1) -- cycle;
        		\draw[color=gray!25] (8,1) -- (13,1) -- (13,3) -- (11,3) -- (11,4) -- (8,4) -- cycle;
        		\draw[color=gray!25] (6,4) -- (11,4) -- (11,6) -- (9,6) -- (9,7) -- (6,7) -- cycle;
        		\draw[color=gray!25] (4,7) -- (9,7) -- (9,9) -- (7,9) -- (7,10) -- (4,10) -- cycle;
            \end{scope}

    		\draw[thick,-to] (-4.5,0) -- (10.5,0) node[anchor=west]{$x$};
    		\draw[thick,-to] (0,-4.5) -- (0,8.5) node[anchor=south]{$y$};
      
            \filldraw[draw=green3,fill=green3!45!white,fill opacity=0.8] (5,-1) -- (10,-1) -- (10,1) -- (8,1) -- (8,2) -- (5,2) -- cycle;
            \filldraw[draw=green3,fill=green3!70!white,fill opacity=0.8] (3,2) -- (8,2) -- (8,4) -- (6,4) -- (6,5) -- (3,5) -- cycle;
            \filldraw[draw=green3,fill=green3!70!white,fill opacity=0.8] (2,-3) -- (7,-3) -- (7,-1) -- (5,-1) -- (5,0) -- (2,0) -- cycle;
            \filldraw[draw=green3,fill=green3!45!white,fill opacity=0.8] (-2,3) -- (3,3) -- (3,5) -- (1,5) -- (1,6) -- (-2,6) -- cycle;
            \filldraw[draw=green3,fill=green3!45!white,fill opacity=0.8] (-3,-2) -- (2,-2) -- (2,0) -- (0,0) -- (0,1) -- (-3,1) -- cycle;
            \filldraw[draw=green3,fill=green3!70!white,fill opacity=0.8] (-5,1) -- (0,1) -- (0,3) -- (-2,3) -- (-2,4) -- (-5,4) -- cycle;
            \filldraw[draw=green3,fill=green3!20!white,fill opacity=0.8] (O) -- (5,0) -- (5,2) -- (3,2) -- (3,3) -- (0,3) -- cycle;

        	\fill[green3!50!black] (P) circle (6pt);
        	\fill[green3!50!black] (Q) circle (6pt);
        	\fill[green3!50!black] (O) circle (6pt);
        	\fill[green3!50!black] (PQ) circle (6pt);

        	\fill[green3!75!black] (-4,6) circle (3pt);
        	\fill[green3!75!black] (-1,8) circle (3pt);
        	\fill[green3!75!black] (-3,-2) circle (3pt);
        	\fill[green3!75!black] (1,5) circle (3pt);
        	\fill[green3!75!black] (2,-3) circle (3pt);
        	\fill[green3!75!black] (4,7) circle (3pt);
        	\fill[green3!75!black] (6,4) circle (3pt);
        	\fill[green3!75!black] (7,-4) circle (3pt);
        	\fill[green3!75!black] (8,1) circle (3pt);
        	\fill[green3!75!black] (9,6) circle (3pt);
    	\end{tikzpicture}
		\caption{Our axis-aligned polytope tiling.}\label{fig:our_tiling}
	\end{subfigure}
	
	\caption{Two tilings of $\mathbb{R}^2$ on the same lattice in which the tiles have area equal to $\det\left(\left[\begin{smallmatrix}5 & -2 \\ -1 & 3\end{smallmatrix}\right]\right) = 13$.}\label{fig:tiling_gives_det}
\end{figure}

For now, instead of placing the usual parallelepiped at each lattice point, we place a cuboid with dimensions $a_{1,1} \times a_{2,2} \times \cdots \times a_{n,n}$, extending in the positive direction. That is, we define the cuboid $C_A$ to be the Cartesian product of closed intervals of lengths given by the diagonal entries of $A$:
\[
    C_A := \prod_{i=1}^n [0, a_{i,i}],
\]
and we consider the set of cuboids $\{C_A + \mathbf{z} : \mathbf{z} \in \Lambda_A\}$. Depending on the values of the off-diagonal entries of $A$, these cuboids may overlap and/or there may be gaps between them (see Figure~\ref{fig:tile_gaps}).

\begin{figure}[!htb]
	\centering
	\begin{subfigure}{.48\textwidth}
		\centering
		\begin{tikzpicture}[scale=0.45]%
    		\coordinate (O) at (0,0);
    		\coordinate (P) at (5,-4);
    		\coordinate (PQ) at (7,-1);
    		\coordinate (Q) at (2,3);
      
            \begin{scope}
                \clip(-4.5,-4.5) rectangle (9.5,8.5);

        		\draw[color=gray!25] (-11,-5) rectangle (-6,-2);
        		\draw[color=gray!25] (-9,-2) rectangle (-4,1);
        		\draw[color=gray!25] (-7,1) rectangle (-2,4);
        		\draw[color=gray!25] (-5,4) rectangle (0,7);
        		\draw[color=gray!25] (-3,7) rectangle (2,10);

        		\draw[color=gray!25] (-4,-6) rectangle (1,-3);
        		\draw[color=gray!25] (-2,-3) rectangle (3,0);
        		\draw[color=gray!25] (2,3) rectangle (7,6);
        		\draw[color=gray!25] (4,6) rectangle (9,9);

        		\draw[color=gray!25] (3,-7) rectangle (8,-4);
        		\draw[color=gray!25] (5,-4) rectangle (10,-1);
        		\draw[color=gray!25] (7,-1) rectangle (12,2);
        		\draw[color=gray!25] (9,2) rectangle (14,5);
        		\draw[color=gray!25] (11,5) rectangle (16,8);
            \end{scope}

    		\draw[thick,-to] (-4.5,0) -- (9.5,0) node[anchor=west]{$x$};
    		\draw[thick,-to] (0,-4.5) -- (0,8.5) node[anchor=south]{$y$};
      
            \filldraw[draw=green3,fill=green3!20!white,fill opacity=0.8] (0,0) rectangle (5,3);

        	\fill[green3!50!black] (P) circle (1pt) node[anchor=east,fill=white,opacity=0.7,draw=white,shift={(-0.2,0)}]{$A\e_1 = (5,-4)$};
        	\fill[green3!50!black] (Q) circle (1pt) node[anchor=south,fill=white,opacity=0.7,draw=white,shift={(0.3,0.05)}]{$A\e_2 = (2,3)$};
        	\fill[green3!50!black] (P) circle (6pt) node[anchor=east,shift={(-0.2,0)}]{$A\e_1 = (5,-4)$};
        	\fill[green3!50!black] (Q) circle (6pt) node[anchor=south,shift={(0.3,0.05)}]{$A\e_2 = (2,3)$};
        	\fill[green3!50!black] (P) circle (6pt);
        	\fill[green3!50!black] (Q) circle (6pt);
        	\fill[green3!50!black] (O) circle (6pt);
        	\fill[green3!50!black] (PQ) circle (6pt);

        	\fill[green3!75!black] (-3,7) circle (3pt);
        	\fill[green3!75!black] (-2,-3) circle (3pt);
        	\fill[green3!75!black] (4,6) circle (3pt);
        	\fill[green3!75!black] (9,2) circle (3pt);
    	\end{tikzpicture}
		\caption{A rectangular packing (with gaps) of $\mathbb{R}^2$ corresponding the matrix $\left[\begin{smallmatrix}5 & 2 \\ -4 & 3\end{smallmatrix}\right]$.}\label{fig:tile_gaps1}
	\end{subfigure} \hfill %
	\begin{subfigure}{.48\textwidth}
		\centering
		\begin{tikzpicture}[scale=0.45]%
    		\coordinate (O) at (0,0);
    		\coordinate (P) at (5,-1);
    		\coordinate (Q) at (-2,3);
    		\coordinate (PQ) at (3,2);
      
            \begin{scope}
                \clip(-4.5,-4.5) rectangle (9.5,8.5);

        		\draw[color=gray!25] (-6,-4) -- (-1,-4) -- (-1,-2) -- (-3,-2) -- (-3,-1) -- (-6,-1) -- cycle;
        		\draw[color=gray!25] (-8,-1) -- (-3,-1) -- (-3,1) -- (-5,1) -- (-5,2) -- (-8,2) -- cycle;
        		\draw[color=gray!25] (-10,2) -- (-5,2) -- (-5,4) -- (-7,4) -- (-7,5) -- (-10,5) -- cycle;

        		\draw[color=gray!25] (-1,-5) -- (4,-5) -- (4,-3) -- (2,-3) -- (2,-2) -- (-1,-2) -- cycle;
        		\draw[color=gray!25] (-3,-2) -- (2,-2) -- (2,0) -- (0,0) -- (0,1) -- (-3,1) -- cycle;
        		\draw[color=gray!25] (-5,1) -- (0,1) -- (0,3) -- (-2,3) -- (-2,4) -- (-5,4) -- cycle;
        		\draw[color=gray!25] (-7,4) -- (-2,4) -- (-2,6) -- (-4,6) -- (-4,7) -- (-7,7) -- cycle;
        		\draw[color=gray!25] (-9,7) -- (-4,7) -- (-4,9) -- (-6,9) -- (-6,10) -- (-9,10) -- cycle;

        		\draw[color=gray!25] (4,-6) -- (9,-6) -- (9,-4) -- (7,-4) -- (7,-3) -- (4,-3) -- cycle;
        		\draw[color=gray!25] (2,-3) -- (7,-3) -- (7,-1) -- (5,-1) -- (5,0) -- (2,0) -- cycle;
        		\draw[color=gray!25] (-2,3) -- (3,3) -- (3,5) -- (1,5) -- (1,6) -- (-2,6) -- cycle;
        		\draw[color=gray!25] (-4,6) -- (1,6) -- (1,8) -- (-1,8) -- (-1,9) -- (-4,9) -- cycle;

        		\draw[color=gray!25] (9,-7) -- (14,-7) -- (14,-5) -- (12,-5) -- (12,-4) -- (9,-4) -- cycle;
        		\draw[color=gray!25] (7,-4) -- (12,-4) -- (12,-2) -- (10,-2) -- (10,-1) -- (7,-1) -- cycle;
        		\draw[color=gray!25] (5,-1) -- (10,-1) -- (10,1) -- (8,1) -- (8,2) -- (5,2) -- cycle;
        		\draw[color=gray!25] (1,5) -- (6,5) -- (6,7) -- (4,7) -- (4,8) -- (1,8) -- cycle;

        		\draw[color=gray!25] (10,-2) -- (15,-2) -- (15,0) -- (13,0) -- (13,1) -- (10,1) -- cycle;
        		\draw[color=gray!25] (8,1) -- (13,1) -- (13,3) -- (11,3) -- (11,4) -- (8,4) -- cycle;
        		\draw[color=gray!25] (6,4) -- (11,4) -- (11,6) -- (9,6) -- (9,7) -- (6,7) -- cycle;
        		\draw[color=gray!25] (4,7) -- (9,7) -- (9,9) -- (7,9) -- (7,10) -- (4,10) -- cycle;
            \end{scope}

    		\draw[thick,-to] (-4.5,0) -- (9.5,0) node[anchor=west]{$x$};
    		\draw[thick,-to] (0,-4.5) -- (0,8.5) node[anchor=south]{$y$};
      
            \filldraw[draw=green3,fill=green3!20!white,fill opacity=0.8] (0,0) rectangle (5,3);

        	\draw[color=gray!50!white,fill=white,opacity=0.5] (3,2) -- (8,2) -- (8,4) -- (6,4) -- (6,5) -- (3,5) -- cycle;

        	\fill[green3!50!black] (P) circle (1pt) node[anchor=north,fill=white,opacity=0.7,draw=white]{$A\e_1 = (5,-1)$};
        	\fill[green3!50!black] (Q) circle (1pt) node[anchor=south,fill=white,opacity=0.7,draw=white,shift={(0,0.05)}]{$A\e_2 = (-2,3)$};
        	\fill[green3!50!black] (P) circle (6pt) node[anchor=north]{$A\e_1 = (5,-1)$};
        	\fill[green3!50!black] (Q) circle (6pt) node[anchor=south,shift={(0,0.05)}]{$A\e_2 = (-2,3)$};

        	\fill[green3!50!black] (O) circle (6pt);
        	\fill[green3!50!black] (PQ) circle (6pt);
        	\fill[green3!75!black] (-4,6) circle (3pt);
        	\fill[green3!75!black] (-1,8) circle (3pt);
        	\fill[green3!75!black] (-3,-2) circle (3pt);
        	\fill[green3!75!black] (1,5) circle (3pt);
        	\fill[green3!75!black] (2,-3) circle (3pt);
        	\fill[green3!75!black] (4,7) circle (3pt);
        	\fill[green3!75!black] (6,4) circle (3pt);
        	\fill[green3!75!black] (7,-4) circle (3pt);
        	\fill[green3!75!black] (8,1) circle (3pt);
        	\fill[green3!75!black] (9,6) circle (3pt);
    	\end{tikzpicture}
		\caption{A rectangular covering (with overlaps) of $\mathbb{R}^2$ corresponding the matrix $\left[\begin{smallmatrix}5 & -2 \\ -1 & 3\end{smallmatrix}\right]$.}\label{fig:tile_gaps2}
	\end{subfigure}
	
	\caption{Two $5 \times 3$ rectangular not-quite-tilings of $\mathbb{R}^2$ coming from matrices with diagonal entries $5$ and $3$. The shaded rectangle is $C_A$, while the other rectangles are its translates on the lattice $\Lambda_A$.}\label{fig:tile_gaps}
\end{figure}
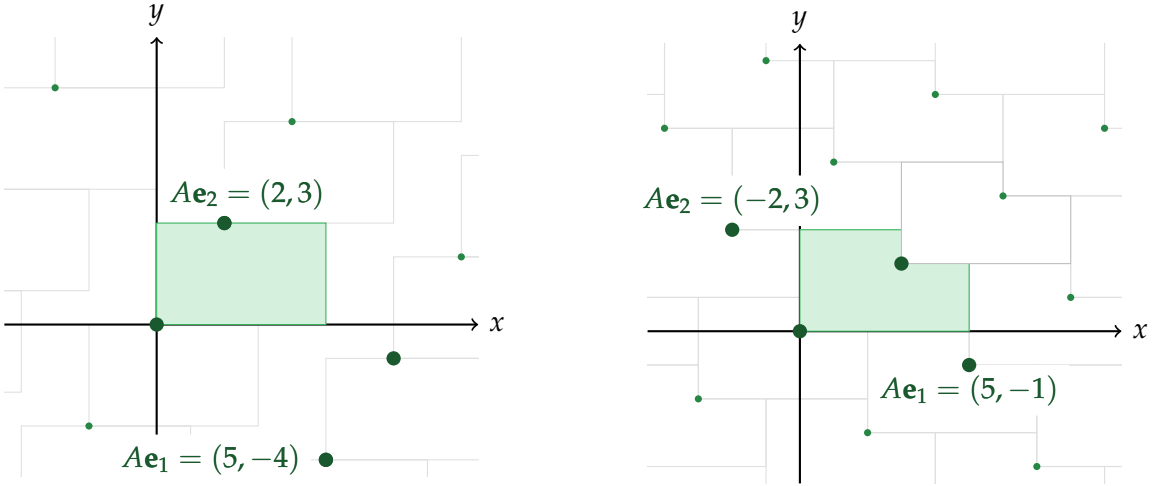

Since these cuboids may overlap and/or have gaps between them, they do not typically form a valid tiling of $\mathbb{R}^n$. In order to ``fix'' the fact that these cuboids can overlap, we define $F_A$ to be the axis-aligned polytope obtained from $C_A$ by removing copies of $C_A$ translated by sums of nonempty subsets of the columns of $A$.\footnote{It may be tempting to remove \emph{all} translates of $C_A$ from $C_A$, but look at Figure~\ref{fig:our_tiling} to see why this does not work; we only want to remove the overlap at the top-right vertex of $C_A$ \emph{or} bottom-left of $C_A$, not both.} We take the closure so that $F_A$ contains its boundary:
\begin{align}\label{eq:FA_defn}
    F_A := \overline{C_A \setminus \bigcup_{\emptyset \neq S \subseteq [n]}\left\{ x + \sum_{i \in S} A \mathbf{e_i} : x \in C_A \right\} }.
\end{align}

In order to similarly ``fix'' the fact that there may be gaps between the cuboids, for the remainder of the proof we only consider matrices $A$ with the following properties:
\begin{itemize}
    \item[(a)] Each diagonal entry $a_{i,i}$ is strictly positive;
    \item[(b)] Each off-diagonal entry, $a_{i,j}$ with $i\neq j$, is strictly negative; and
    \item[(c)] Each row has strictly positive sum, i.e., $A$ is strictly diagonally dominant.
\end{itemize}
Doing so simplifies the rest of the proof considerably (since the determinant for these matrices is strictly positive, so it equals a volume instead of a signed volume, for example), and it does not result in any loss of generality. Indeed, if two polynomials agree on some open set then they must agree everywhere; the definitional formula of the determinant from Equation~\eqref{eq:det_per_defn} and the formula described by Theorem~\ref{thm:main_formula} are both polynomials in the $n^2$ entries of the matrix $A$, so if they agree on some open set (like the set of matrices described by conditions (a)--(c), or even just the smaller set described by the upcoming Lemma~\ref{lem:FA_overlaps}) then they must agree everywhere.

In the next section we establish that, subject to these conditions (a)--(c), the translates of $F_A$ by vectors in $\Lambda_A$ tile the ambient space $\mathbb{R}^n$: their union covers all of space, and any two distinct translates intersect on a set of measure zero (i.e. on their boundary or not at all). We also show that two tiles $F_A + A\mathbf{u}$ and $F_A + A\mathbf{v}$ share a common boundary if and only if the difference $\mathbf{u} - \mathbf{v}$ is the sum of some non-empty proper subset of the vectors $\{ \mathbf{e_0}, \mathbf{e_1}, \ldots, \mathbf{e_n} \}$, in which case we describe these two translates as ``neighbours''. This is equivalent to $\mathbf{u} - \mathbf{v}$ being a nonzero vector with all entries in $\{0, 1\}$ or all entries in $\{0, -1\}$. Each tile has exactly $2(2^n - 1)$ neighbours.

This tiling admits a proper $(n+1)$-colouring, where $F_A + A\mathbf{v}$ is coloured according to the sum of the coordinates of $\mathbf{v}$ modulo $n+1$. Observe that if two tiles $F_A + A\mathbf{u}$ and $F_A + A\mathbf{v}$ are neighbours, then the coordinate sum of $\mathbf{u} - \mathbf{v}$ is in $[-n, -1] \cup [1, n]$ and is therefore nonzero modulo $n + 1$, so the tiles are assigned distinct colours. Figure~\ref{fig:our_tiling} shows $F_A$ and its neighbours coloured in this manner.

Moreover, the tiling is homeomorphic to the standard permutohedral tiling of $\mathbb{R}^n$ obtained by taking the Voronoi tessellation of the lattice $A_n^{\star}$ defined in \cite{CS88}. For $n = 2$, this is the familiar hexagonal tessellation; for $n = 3$, this is the tessellation by truncated octahedra whose centres form the body-centred cubic lattice.

\subsection{The Proof Itself}\label{sec:geo_proof_subsec}

For an arbitrary $n \times n$ matrix $A$, we shall consider five different propositions (each of which may be true or false, depending on $A$):

\begin{itemize}
    \item $P_1(A)$: $\det(A)$ is equal to the formula described by Equation~\eqref{eq:main_formula}.
    \item $P_2(A)$: the translates $\{F_A + \mathbf{z} : \mathbf{z} \in \Lambda_A\}$ have pairwise measure-zero intersections.
    \item $P_3(A)$: the translates $\{F_A + \mathbf{z} : \mathbf{z} \in \Lambda_A\}$ cover all of $\mathbb{R}^n$.
    \item $P_4(A)$: the volume of $F_A$ is equal to the formula described by Equation~\eqref{eq:main_formula}.
    \item $P_5(A)$: the volume of $F_A$ is equal to $\det(A)$.
\end{itemize}

Any two of $P_2(A), P_3(A), P_5(A)$ together imply the third (being equivalent to the claim that the translates of $F_A$ by the vectors in the lattice $\Lambda_A$ tile space). Also, any two of $P_1(A), P_4(A), P_5(A)$ together imply the third (by transitivity of equality). We summarise these relationships in Figure~\ref{fig:P15_diagram}.

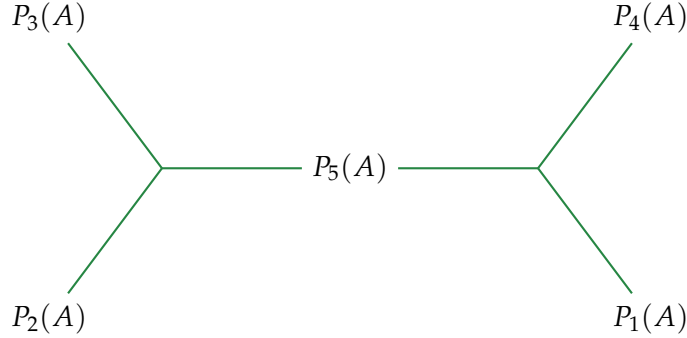
\begin{figure}[!htb]
    \centering
    \begin{tikzpicture}
        \node(P1) at (4, -2) {$P_1(A)$};
        \node(P2) at (-4, -2) {$P_2(A)$};
        \node(P3) at (-4, 2) {$P_3(A)$};
        \node(P4) at (4, 2) {$P_4(A)$};
        \node(P5) at (0, 0) {$P_5(A)$};
        \draw[thick,green3!75!black] (P1) -- (2.5, 0);
        \draw[thick,green3!75!black] (P4) -- (2.5, 0);
        \draw[thick,green3!75!black] (P5) -- (2.5, 0);
        \draw[thick,green3!75!black] (P2) -- (-2.5, 0);
        \draw[thick,green3!75!black] (P3) -- (-2.5, 0);
        \draw[thick,green3!75!black] (P5) -- (-2.5, 0);
    \end{tikzpicture}
    \caption{A summary of the relationship between the properties $P_1(A)$--$P_5(A)$. Any two of the properties in the left Y-shape imply the third, and any two of the properties in the right Y-shape imply the third.}\label{fig:P15_diagram}
\end{figure}

Theorem~\ref{thm:main_formula} is equivalent to $P_1(A)$ being true for all matrices $A$. The proof proceeds by showing that $P_3(A)$ and $P_4(A)$ are true for all matrices $A$ satisfying the three conditions (a)--(c), and that $P_2(A)$ is true for all matrices $A$ in a small open neighbourhood of a canonical matrix $B$. We deduce that $P_1(A)$ holds on a non-empty open set, and therefore (being an equality between two polynomials) holds in general, completing the proof of Theorem~\ref{thm:main_formula}. Note that, by traversing the implications in the other direction, it follows that $P_5(A)$ and $P_2(A)$ are true for all matrices $A$ satisfying conditions (a)--(c), and not just those in the small neighbourhood for which we prove $P_2(A)$ directly.

\begin{lemma}\label{lem:FA_overlaps}
    Consider the $n \times n$ matrix
    \begin{align}\label{eq:canon_matrix_B}
        B := (n+1)I - J,
    \end{align}
    where $J$ is the matrix with all entries equal to $1$. There is some $\varepsilon$-neighbourhood $N_\varepsilon$ of $B$ with the property that, for all $A \in N_\varepsilon$ and all $\mathbf{y} \neq \mathbf{z} \in \Lambda_A$, it is the case that $(F_A + \mathbf{y}) \cap (F_A + \mathbf{z})$ has measure zero (i.e., any two tiles in the tiling overlap only on their boundary or not at all).
\end{lemma}

Before proving the above lemma, we note that we will choose the $\varepsilon$-neighbourhood $N_\varepsilon$ to be small enough that every $A \in N_\varepsilon$ satisfies the three conditions (a)--(c) that we described earlier (which is possible since $B$ satisfies those three conditions and they define an open set).

$B$ is a positive-definite symmetric matrix with eigenvalues $1$ (with multiplicity $1$) and $n+1$ (with multiplicity $n-1$). The singular values of $B$ are equal to its eigenvalues, so the minimum singular value of $B$ is also equal to $1$. By choosing $\varepsilon$ small enough, we can ensure that every $A \in N_\varepsilon$ has minimum singular value greater than $\frac{1}{2}$.

For concreteness, let $\varepsilon_0 > 0$ be a constant (depending only on $n$) that ensures that every $A \in N_\varepsilon$ satisfies conditions (a)--(c) and has minimum singular value greater than $\frac{1}{2}$.

\begin{proof}[Proof of Lemma~\ref{lem:FA_overlaps}]
    It suffices to prove that $F_A \cap (F_A + \mathbf{z})$ has measure zero whenever $\mathbf{z} \in \Lambda_A$ is non-zero. Since $A$ is strictly diagonally dominant and thus invertible, $\mathbf{z} \in \Lambda_A$ being non-zero is equivalent to $\mathbf{z} = A\mathbf{v}$ for some non-zero $\mathbf{v} \in \mathbb{Z}^n$.

    If $\mathbf{v}$ is non-zero and has all entries in $\{0,1\}$ or all entries in $\{0,-1\}$ then, by the construction of $F_A$ given in Equation~\eqref{eq:FA_defn}, we know that $F_A \cap (F_A + A\mathbf{v})$ has measure zero (we call these $2(2^n - 1)$ tiles $F_A + A\mathbf{v}$ the ``neighbours'' of $F_A$; when $n = 2$ they are exactly the $6$ tiles that touch the central shaded tile in Figure~\ref{fig:our_tiling}). Our goal now is to show that every non-neighbour of $F_A$ has empty intersection with $F_A$.
    
    To this end, first consider the matrix $B$ from Equation~\eqref{eq:canon_matrix_B}. For this matrix, the polytope under consideration is
    \[
        F_B = \big\{ (x_1,x_2,\ldots,x_n) : 0 \leq x_i^\uparrow \leq i \ \text{for all} \ i \in [n] \big\},
    \]
    where $x_i^\uparrow$ is the $i$-th smallest entry of $(x_1,x_2,\ldots,x_n)$ (see Figure~\ref{fig:canon_matrix_3d}).  Equivalently, $F_B$ is the union of the $n!$ images of the cuboid $[0, 1] \times [0, 2] \times \cdots \times [0, n]$ under arbitrary permutations of the coordinates. We claim that $F_B$ does not intersect $F_B + B\mathbf{v}$ if $F_B + B\mathbf{v}$ is not a neighbour of $F_B$.
    
    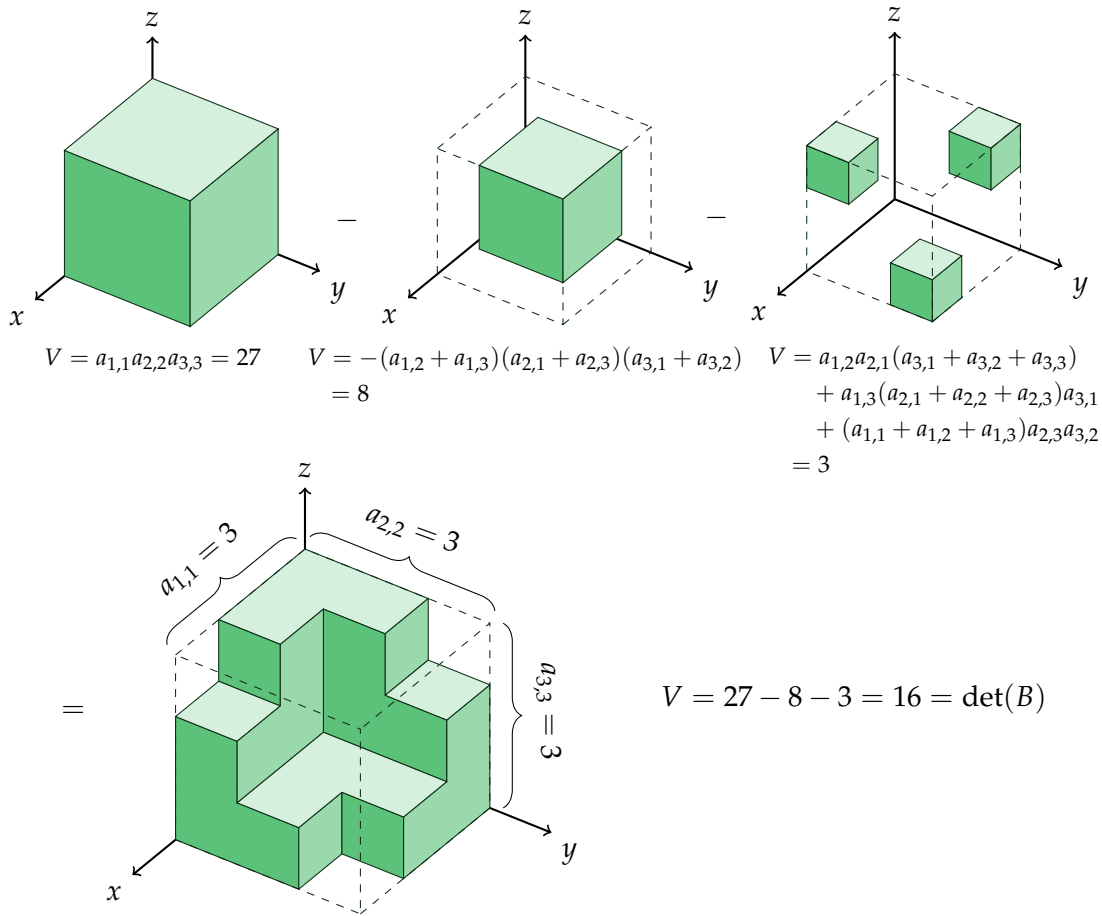
\begin{figure}[!htb]
    \centering
	\tdplotsetmaincoords{55}{125}\begin{tikzpicture}[scale=0.68,tdplot_main_coords]
		\draw[thick,-to] (0,0,0) -- (4,0,0) node[anchor=north east]{$x$};
		\draw[thick,-to] (0,0,0) -- (0,4,0) node[anchor=north west]{$y$};
		\draw[thick,-to] (0,0,0) -- (0,0,4) node[anchor=south]{$z$};
		
        \draw[draw=green3!25!black,fill=green3!80!white] (3,0,0) -- (3,0,3) -- (3,3,3) -- (3,3,0) -- cycle;
        \draw[draw=green3!25!black,fill=green3!50!white] (0,3,0) -- (0,3,3) -- (3,3,3) -- (3,3,0) -- cycle;
        \draw[draw=green3!25!black,fill=green3!20!white] (0,0,3) -- (3,0,3) -- (3,3,3) -- (0,3,3) -- cycle;
        
        \node[anchor=north,align=left] at (0,0,-3.2){\footnotesize $V = a_{1,1}a_{2,2}a_{3,3} = 27$\\ ${}$\\ ${}$\\ ${}$};
	\end{tikzpicture}\kern-2em
	\begin{tikzpicture}[scale=0.68,tdplot_main_coords]
		\draw[thick,-to] (0,0,0) -- (4,0,0) node[anchor=north east]{$x$};
		\draw[thick,-to] (0,0,0) -- (0,4,0) node[anchor=north west]{$y$};
		\draw[thick,-to] (0,0,0) -- (0,0,4) node[anchor=south]{$z$};

        \draw[green3!25!black,dashed] (3,0,0) -- (3,0,3) -- (3,3,3) -- (3,3,0) -- cycle;
        \draw[green3!25!black,dashed] (3,3,0) -- (0,3,0) -- (0,3,3) -- (3,3,3);
        \draw[green3!25!black,dashed] (3,0,3) -- (0,0,3) -- (0,3,3);
        
        \draw[draw=green3!25!black,fill=green3!80!white] (3,1,1) -- (3,1,3) -- (3,3,3) -- (3,3,1) -- cycle;
        \draw[draw=green3!25!black,fill=green3!50!white] (1,3,1) -- (1,3,3) -- (3,3,3) -- (3,3,1) -- cycle;
        \draw[draw=green3!25!black,fill=green3!20!white] (1,1,3) -- (3,1,3) -- (3,3,3) -- (1,3,3) -- cycle;
        
        \node at (2.5,-2.5,0){$-$};
        \node[anchor=north,align=left] at (0,0,-3.2){\footnotesize $V = -(a_{1,2} + a_{1,3})(a_{2,1} + a_{2,3})(a_{3,1} + a_{3,2})$\\\footnotesize \quad $ = 8$\\ ${}$\\ ${}$};
	\end{tikzpicture}\kern-2em
	\begin{tikzpicture}[scale=0.68,tdplot_main_coords]
		\draw[thick,-to] (0,0,0) -- (4,0,0) node[anchor=north east]{$x$};
		\draw[thick,-to] (0,0,0) -- (0,4,0) node[anchor=north west]{$y$};
		\draw[thick,-to] (0,0,0) -- (0,0,4) node[anchor=south]{$z$};
        
        \draw[draw=green3!25!black,fill=green3!80!white] (1,2,2) -- (1,2,3) -- (1,3,3) -- (1,3,2) -- cycle;
        \draw[draw=green3!25!black,fill=green3!50!white] (0,3,2) -- (0,3,3) -- (1,3,3) -- (1,3,2) -- cycle;
        \draw[draw=green3!25!black,fill=green3!20!white] (1,2,3) -- (0,2,3) -- (0,3,3) -- (1,3,3) -- cycle;
        
        \draw[draw=green3!25!black,fill=green3!50!white] (2,1,2) -- (2,1,3) -- (3,1,3) -- (3,1,2) -- cycle;
        \draw[draw=green3!25!black,fill=green3!80!white] (3,0,2) -- (3,0,3) -- (3,1,3) -- (3,1,2) -- cycle;
        \draw[draw=green3!25!black,fill=green3!20!white] (2,1,3) -- (2,0,3) -- (3,0,3) -- (3,1,3) -- cycle;
        
        \draw[draw=green3!25!black,fill=green3!20!white] (2,2,1) -- (2,3,1) -- (3,3,1) -- (3,2,1) -- cycle;
        \draw[draw=green3!25!black,fill=green3!80!white] (3,2,0) -- (3,3,0) -- (3,3,1) -- (3,2,1) -- cycle;
        \draw[draw=green3!25!black,fill=green3!50!white] (2,3,1) -- (2,3,0) -- (3,3,0) -- (3,3,1) -- cycle;
		
        \draw[green3!25!black,dashed] (3,0,0) -- (3,0,3) -- (3,3,3) -- (3,3,0) -- cycle;
        \draw[green3!25!black,dashed] (3,3,0) -- (0,3,0) -- (0,3,3) -- (3,3,3);
        \draw[green3!25!black,dashed] (3,0,3) -- (0,0,3) -- (0,3,3);
        
        \node at (2.5,-2.5,0){$-$};
        \node[anchor=north,align=left] at (0.1,1,-2.8){\footnotesize $V = a_{1,2}a_{2,1}(a_{3,1} + a_{3,2} + a_{3,3})$\\\footnotesize \qquad $+ \ a_{1,3}(a_{2,1} + a_{2,2} + a_{2,3})a_{3,1}$\\\footnotesize \qquad $+ \ (a_{1,1} + a_{1,2} + a_{1,3})a_{2,3}a_{3,2}$\\\footnotesize \quad $ = 3$};
	\end{tikzpicture}\\[-1.1em]
	\begin{tikzpicture}[scale=1,tdplot_main_coords]
		\draw[thick,-to] (0,0,0) -- (4,0,0) node[anchor=north east]{$x$};
		\draw[thick,-to] (0,0,0) -- (0,4,0) node[anchor=north west]{$y$};
		\draw[thick,-to] (0,0,0) -- (0,0,4) node[anchor=south]{$z$};
		
        \draw[draw=green3!25!black,fill=green3!80!white] (3,0,0) -- (3,0,2) -- (3,1,2) -- (3,1,1) -- (3,2,1) -- (3,2,0) -- cycle;
        \draw[draw=green3!25!black,fill=green3!50!white] (0,3,0) -- (0,3,2) -- (1,3,2) -- (1,3,1) -- (2,3,1) -- (2,3,0) -- cycle;
        \draw[draw=green3!25!black,fill=green3!20!white] (0,0,3) -- (0,2,3) -- (1,2,3) -- (1,1,3) -- (2,1,3) -- (2,0,3) -- cycle;
        \draw[draw=green3!25!black,fill=green3!80!white] (2,0,2) -- (2,0,3) -- (2,1,3) -- (2,1,2) -- cycle;
        \draw[draw=green3!25!black,fill=green3!20!white] (3,0,2) -- (2,0,2) -- (2,1,2) -- (3,1,2) -- cycle;
        \draw[draw=green3!25!black,fill=green3!50!white] (0,2,2) -- (0,2,3) -- (1,2,3) -- (1,2,2) -- cycle;
        \draw[draw=green3!25!black,fill=green3!20!white] (0,3,2) -- (0,2,2) -- (1,2,2) -- (1,3,2) -- cycle;
        \draw[draw=green3!25!black,fill=green3!50!white] (3,2,0) -- (2,2,0) -- (2,2,1) -- (3,2,1) -- cycle;
        \draw[draw=green3!25!black,fill=green3!80!white] (2,2,0) -- (2,3,0) -- (2,3,1) -- (2,2,1) -- cycle;
        \draw[draw=green3!25!black,fill=green3!50!white] (3,1,1) -- (3,1,2) -- (2,1,2) -- (2,1,3) -- (1,1,3) -- (1,1,1) -- cycle;
        \draw[draw=green3!25!black,fill=green3!80!white] (1,3,1) -- (1,3,2) -- (1,2,2) -- (1,2,3) -- (1,1,3) -- (1,1,1) -- cycle;
        \draw[draw=green3!25!black,fill=green3!20!white] (3,1,1) -- (3,2,1) -- (2,2,1) -- (2,3,1) -- (1,3,1) -- (1,1,1) -- cycle;
        
        \draw[green3!25!black,dashed] (3,0,0) -- (3,0,3) -- (3,3,3) -- (3,3,0) -- cycle;
        \draw[green3!25!black,dashed] (3,3,0) -- (0,3,0) -- (0,3,3) -- (3,3,3);
        \draw[green3!25!black,dashed] (3,0,3) -- (0,0,3) -- (0,3,3);

        \draw [decorate, decoration={brace,raise=5pt,amplitude=7pt}] (3,0,3) --node[anchor=south,rotate=40,shift={(-0.6,-0.5)}]{$a_{1,1} = 3$} (0,0,3);
        \draw [decorate, decoration={brace,raise=5pt,amplitude=7pt}] (0,0,3) --node[anchor=south,rotate=-20,shift={(-0.6,-0.5)}]{$a_{2,2} = 3$} (0,3,3);
        \draw [decorate, decoration={brace,raise=5pt,amplitude=7pt}] (0,3,3) --node[anchor=south,rotate=-90,shift={(-0.6,-0.5)}]{$a_{3,3} = 3$} (0,3,0);
        
        \node[anchor=east] at (2,-2,0.7){$=$};
        \node[anchor=west] at (-3.3,3.3,0){$V = 27 - 8 - 3 = 16 = \det(B)$};
	\end{tikzpicture}
 
    \caption{The polytope $F_B$ for the matrix $B = (n+1)I - J$, with $n = 3$. Its volume is equal to the sum and difference of the volumes of $5$ different cubes, corresponding to the $5$-term formula for the determinant~\eqref{eq:det_3x3}.}\label{fig:canon_matrix_3d}
\end{figure}

To prove this claim, suppose that $F_B + B\mathbf{v}$ is not a neighbour
of $F_B$.

Firstly, consider the case where there exist indices $i$ and $j$ such that
$\mathbf{v_i} - \mathbf{v_j} \geq 2$. Then $(B\mathbf{v})_i - (B\mathbf{v})_j
= (n + 1)(\mathbf{v_i} - \mathbf{v_j}) \geq 2(n+1)$. Consequently, one of
$(B\mathbf{v})_i$ and $(B\mathbf{v})_j$ has absolute value $\geq n + 1$, which means
that the bounding cubes of $F_B + B\mathbf{v}$ and $F_B$ (which each have
sidelength $n$) are disjoint.

This leaves the case where all entries of $\mathbf{v}$ differ by at most 1.
In other words, there exists $c \in \mathbb{Z}$ such that for all $i$,
we have $\mathbf{v_i} \in \{c-1, c\}$, and moreover there exists at least
one such $j$ such that $\mathbf{v_j} = c$.

We can assume without loss of generality that $c$ is positive (because
$F_B + B\mathbf{v}$ is disjoint from $F_B$ if and only if $F_B - B\mathbf{v}$
is disjoint from $F_B$). Moreover, if $c \in \{0, 1\}$ then the polytopes
are neighbours, so we have $c \geq 2$.

By permuting the coordinates, we can assume that $\mathbf{v_1}, \dots, \mathbf{v_k} = c$ and
$\mathbf{v_{k+1}}, \dots, \mathbf{v_n} = c - 1$. Then we can compute that:

\[ (B\mathbf{v})_1 = \cdots = (B\mathbf{v})_k = n - k + c \]
and, because $c \geq 2$, this means that at least $k$ of the entries
$(B\mathbf{v})_i \geq n - k + 2$. As such, $B\mathbf{v} \notin F_B$, because every
point $x \in F_B$ has at least $n - k + 1$ entries $x_i \leq n - k + 1$.
Moreover, because $F_B$ is contained in the positive orthant, we also
have $x + B\mathbf{v} \notin F_B$ for all $x \in F_B$, establishing that
the two polytopes are disjoint.

    Since $F_B$ has empty intersection with $F_B + B\mathbf{v}$ whenever they are non-neighbours, and $B$ has integer entries, the distance between non-neighbours $F_B$ and $F_B + B\mathbf{v}$ must be at least $1$. Since the coordinates of $F_A$ and $F_A + A\mathbf{v}$ are continuous in the entries of $A$, we conclude that there is some $\varepsilon$-neighbourhood $N_\varepsilon$ of $B$ with the property that $F_A \cap (F_A + A\mathbf{v}) = \emptyset$ whenever $A \in N_\varepsilon$ and $F_A + A\mathbf{v}$ is not a neighbour of $F_A$.
    
    However, the particular choice of $\varepsilon$ can depend on the vector $\mathbf{v}$, so we denote it by $\varepsilon_{\mathbf{v}}$. To fix this problem, note that if $\|\mathbf{v}\| \geq 2K$ then $\|A\mathbf{v}\| > K$ (since the minimum singular value of $A$ is at least $\frac{1}{2}$), so there are only finitely many $\mathbf{v} \in \mathbb{Z}^n$ for which $F_A \cap (F_A + A\mathbf{v})$ is potentially non-empty: choose $K$ to be an upper bound on the diameter of $F_A$; the only $\mathbf{v} \in \mathbb{Z}^n$ that need to be considered are those in the ball of radius $2K$. Defining $\varepsilon$ to be the minimum of $\varepsilon_0$ (defined at the beginning of this proof) and these finitely many $\varepsilon_{\mathbf{v}}$'s completes the proof.
\end{proof}

\begin{lemma}\label{lem:FA_covers}
    If $A$ satisfies conditions (a)--(c) above then $F_A + \Lambda_A = \R^n$ (i.e., there are no gaps in the tiling).
\end{lemma}

\begin{proof}
    We first note that it suffices to show that, for each $\mathbf{x} \in \R^n$, there exists $\mathbf{z} \in \Lambda_A$ such that $\mathbf{x} + \mathbf{z} \in C_A$. To see this, recall from Equation~\eqref{eq:FA_defn} that $F_A$ is constructed from $C_A$ by removing its immediate neighbours in the positive direction, so if $\mathbf{x} + \mathbf{z} \in C_A$ then $\mathbf{x} \in (F_A + A\mathbf{b}) - \mathbf{z}$ for some binary vector $\mathbf{b} \in \{0,1\}^n$. Since $A\mathbf{b} - \mathbf{z} \in \Lambda_A$, this implies $\mathbf{x} \in F_A + \Lambda_A$.

    To find $\mathbf{z}$, first pick some $\mathbf{z^{(0)}} \in \Lambda_A$ such that each entry of $\mathbf{y^{(0)}} := \mathbf{x} + \mathbf{z^{(0)}}$ is non-negative (such a $\mathbf{z^{(0)}}$ exists since $A$ is strictly diagonally dominant and thus invertible). Set $k = 0$ and proceed inductively as follows:

    \begin{itemize}
        \item[(i)] If $\mathbf{y}^{\mathbf{(k)}}_i \leq a_{i,i}$ for all $1 \leq i \leq n$ then $\mathbf{y^{(k)}} \in C_A$, so we can choose $\mathbf{z} = \mathbf{z^{(k)}}$ and be done.
        
        \item[(ii)] Otherwise, pick an index $1 \leq i \leq n$ for which $\mathbf{y}^{\mathbf{(k)}}_i > a_{i,i}$ and set $\mathbf{y}^{\mathbf{(k+1)}}_i = \mathbf{y}^{\mathbf{(k)}}_i - A\mathbf{e_i}$ and $\mathbf{z}^{\mathbf{(k+1)}}_i = \mathbf{z}^{\mathbf{(k)}}_i - A\mathbf{e_i}$. Increase $k$ by $1$ and then repeat these bullet points.
    \end{itemize}

    Since the diagonal entries of $A$ are strictly positive, it is clear that $0 \leq \mathbf{y}^{\mathbf{(k+1)}}_i < \mathbf{y}^{\mathbf{(k)}}_i$. However, decreasing the $i$-th entry like this comes at the expense of increasing the other entries (since the off-diagonal entries of $A$ are negative). It is thus not obvious that the inductive procedure described above terminates. To see that it does, we demonstrate the existence of a vector $\mathbf{v} \in \R^n$ and a scalar $0 < d \in \R$ with the properties that $\mathbf{v} \cdot \mathbf{y}^{\mathbf{(k)}} \geq 0$ for all $k$ and $(\mathbf{v} \cdot \mathbf{y}^{\mathbf{(k)}}) - (\mathbf{v} \cdot \mathbf{y}^{\mathbf{(k+1)}}) \geq d$ for all $k$, implying that the procedure terminates for some $k \leq (\mathbf{v} \cdot \mathbf{y}^{\mathbf{(0)}}) / d$.
    
    To construct such a $\mathbf{v}$ and $d$, let $c \in \R$ be large enough that $cI - A^T$ has all entries strictly positive. The Perron--Frobenius theorem tells us that there is a strictly positive real eigenvalue $\lambda$ with a corresponding entrywise strictly positive eigenvector $\mathbf{v}$ such that $(cI - A^T)\mathbf{v} = \lambda\mathbf{v}$. Since $\mathbf{v}$ and $\mathbf{y}^{\mathbf{(k)}}$ are both entrywise non-negative, we have $\mathbf{v} \cdot \mathbf{y}^{\mathbf{(k)}} \geq 0$ for all $k$. Furthermore, since $(c-\lambda,\mathbf{v})$ is also an eigenvalue-eigenvector pair of $A^T$, and $A$ is strictly diagonally dominant (and $c$ and $\lambda$ are both real), we know that $c - \lambda > 0$. It follows that
    \[
        \mathbf{v} \cdot (A\mathbf{e_i}) = (A^T\mathbf{v}) \cdot \mathbf{e_i} = (c - \lambda)\mathbf{v} \cdot \mathbf{e_i} = (c-\lambda)v_i > 0 \quad \text{for all} \quad 1 \leq i \leq n.
    \]
    If we choose $d := (c-\lambda)\min_i\{v_i\} > 0$ then it follows that $(\mathbf{v} \cdot \mathbf{y}^{\mathbf{(k)}}) - (\mathbf{v} \cdot \mathbf{y}^{\mathbf{(k+1)}}) = \mathbf{v} \cdot (A\mathbf{e_i}) \geq d$, which completes the proof.
\end{proof}

\begin{lemma}\label{lem:FA_volume}
    If $A$ satisfies conditions (a)--(c) above then the volume of $F_A$ is given by the expression in Equation~\eqref{eq:main_formula}.
\end{lemma}

\begin{proof}
    The volume of the cuboid $C_A$ is clearly equal to $a_{1,1}a_{2,2}\cdots a_{n,n}$, which is one of the terms in the sum~\eqref{eq:main_formula} (it is the term corresponding to the empty partial partition). We now use inclusion-exclusion to show that the rest of the terms in that sum correspond to volumes that were removed by translations of $C_A$ when creating $F_A$ as in Equation~\eqref{eq:FA_defn}.

    For a non-empty $S \subseteq [n]$, the set
    \begin{align}\label{eq:cuboid_intersection}
        C_A \cap \left(C_A + \sum_{i \in S}A\mathbf{e_i}\right)
    \end{align}
    is a cuboid with its $i$-th side length $\ell_i$ equal to
    \[
        \ell_i = \begin{cases}
            \displaystyle a_{i,i}-\sum_{j \in S}a_{i,j}, & \text{if $i \in S$} \\
            \displaystyle a_{i,i} + \sum_{j \in S}a_{i,j}, & \text{if $i \notin S$}
        \end{cases} = \begin{cases}
            \displaystyle -\sum_{j \in S,j \neq i}a_{i,j}, & \text{if $i \in S$} \\
            \displaystyle a_{i,i} + \sum_{j \in S}a_{i,j}, & \text{if $i \notin S$}.
        \end{cases}.
    \]
    Multiplying these side lengths together gives us the volume of the cuboid~\eqref{eq:cuboid_intersection}. Subtracting this quantity for all non-empty $S \subseteq [n]$ (i.e., subtracting the volume of all of these cuboids that are removed from $C_A$ to create $F_A$) results in the following (not yet correct) formula for the volume of $F_A$:
    \begin{align}\label{eq:det_formula_geometry_inc}
        a_{1,1}a_{2,2}\cdots a_{n,n} - \sum_{\emptyset \neq S \subseteq [n]}\prod_{i=1}^n \ell_i = a_{1,1}a_{2,2}\cdots a_{n,n} - \sum_{\emptyset \neq S \subseteq [n]}\prod_{i=1}^n\begin{cases}
            \displaystyle -\sum_{j \in S,j \neq i}a_{i,j}, & \text{if $i \in S$} \\
            \displaystyle a_{i,i} + \sum_{j \in S}a_{i,j}, & \text{if $i \notin S$}.
        \end{cases}
    \end{align}

    When $n \leq 3$ the formula~\eqref{eq:det_formula_geometry_inc} is the same as the formula~\eqref{eq:main_formula} and indeed equals the volume of $F_A$. In particular, if $n = 2$ then it expresses the area of the shaded tile $F_A$ in Figure~\ref{fig:our_tiling} as the area $a_{1,1}a_{2,2}$ of the rectangle $C_A$ from Figure~\ref{fig:tile_gaps2} minus the area of the rectangular overlapping region at its top-right vertex, and if $n = 3$ then it expresses the volume of $F_A$ as the volume $a_{1,1}a_{2,2}a_{3,3}$ of the cuboid $C_A$ minus the volumes of $4$ other cuboids (these cuboids are depicted in the case of the matrix $(n+1)I - J$ in Figure~\ref{fig:canon_matrix_3d}, and the picture for other matrices satisfying (a)--(c) is similar).

    When $n \geq 4$, however, the formula~\eqref{eq:det_formula_geometry_inc} is not quite correct, since there are overlaps-of-overlaps of the translates of $C_A$, so some volume that is removed from $C_A$ in Equation~\eqref{eq:det_formula_geometry_inc} is removed multiple times. To correct this mistake, we proceed via inclusion-exclusion: we add back the volumes that were subtracted too many times, then we subtract volumes that were added back too many times, and so on.

    For example, when $n = 4$, the cuboids corresponding to the subsets $S_1 = \{1,2\}$, and $S_2 = \{1,2,3,4\}$ (i.e., the cuboids $C_A + A(\mathbf{e_1} + \mathbf{e_2})$ and $C_A + A(\mathbf{e_1} + \mathbf{e_2} + \mathbf{e_3} + \mathbf{e_4})$) overlap with each other, so formula~\eqref{eq:det_formula_geometry_inc} is too small: it subtracts off the volume of the cuboid
    \[
        (C_A + A(\mathbf{e_1} + \mathbf{e_2})) \cap (C_A + A(\mathbf{e_1} + \mathbf{e_2} + \mathbf{e_3} + \mathbf{e_4}))
    \]
    twice (this volume equals $a_{1,2}a_{2,1}a_{3,4}a_{4,3}$). In fact, there are exactly six cuboids whose volumes were subtracted off twice by formula~\eqref{eq:det_formula_geometry_inc}, given by $S_2 = \{1,2,3,4\}$ and $S_1$ being any $2$-element subset of $S_2$. To correct this mistake, we simply add the volumes of these cuboids back in. The cuboids corresponding to $S_1 = \{1,2\}$ and $S_1 = \{3,4\}$ have the same volume as each other, as do the cuboids corresponding to $S_1 = \{1,3\}$ and $S_1 = \{2,4\}$, as do the cuboids corresponding to $S_1 = \{1,4\}$ and $S_1 = \{2,3\}$, thus giving the final three terms (each with a coefficient of $2$) in Equation~\eqref{eq:det_n4_formula}.

    In general, $k$ pairwise distinct (but not necessarily disjoint) subsets $S_1,S_2,\ldots,S_k \subseteq [n]$ are such that
    \begin{align}\label{eq:arb_intersection_S1S2}
        \bigcap_{j=1}^k \left(C_A + \sum_{i \in S_j}A\mathbf{e_i}\right)
    \end{align}
    has non-zero volume if and only if there is a chain of inclusions among them, which we will assume is $S_1 \subset S_2 \subset \cdots \subset S_k \subseteq [n]$ without loss of generality, and $|S_\ell \setminus S_{\ell-1}| \geq 2$ for all $1 \leq \ell \leq k$ (we define $S_0 = \emptyset$ for convenience). Indeed, if $|S_\ell \setminus S_{\ell-1}| = 0$ then $S_\ell = S_{\ell-1}$, and if $|S_\ell \setminus S_{\ell-1}| = 1$ then the two cuboids corresponding to $j = \ell-1$ and $j = \ell$ in Equation~\eqref{eq:arb_intersection_S1S2} intersect only on their boundary ($S_\ell \setminus S_{\ell-1} = \{i\}$ for some $i \in [n]$, so the cuboids are offset from each other by $a_{i,i}$ in the direction of the $i$-th coordinate axis, which is also their width in that direction).

    The volume of the cuboid described by the intersection in Equation~\eqref{eq:arb_intersection_S1S2} is the product of its side lengths, which equals
    \begin{align}\label{eq:arb_intersection_Sk}
        \prod_{i=1}^n\begin{cases}
            \displaystyle  - \ \ \ \smashoperator{\sum_{j \in S_\ell \setminus S_{\ell-1},j \neq i}} \ \ \ a_{i,j}, & \text{if $i \in S_\ell \setminus S_{\ell-1}$ ($1 \leq \ell \leq k$)} \\
            \displaystyle a_{i,i} + \sum_{j \in S_k}a_{i,j}, & \text{if $i \notin S_k$}.
        \end{cases}
    \end{align}
    The result now follows from using inclusion-exclusion, associating the chain $S_1 \subset S_2 \subset \cdots \subset S_k \subseteq [n]$ with the partial partition
    \[
        P = \{S_1, S_2 \setminus S_1, S_3 \setminus S_2, \ldots, S_k \setminus S_{k-1}\},
    \]
    and noting that there are exactly $k!$ different chains $S_1 \subset S_2 \subset \cdots \subset S_k \subseteq [n]$ that give rise to this particular partial partition (since the order of the parts in $P$ does not matter).\footnote{For example, this partial partition has $|P| = k$ and $\sgn(P) = (-1)^{|S_k|+k}$, with the $(-1)^k$ coming from the fact that $k$ specifies which level of the inclusion-exclusion we are at, and $(-1)^{|S_k|}$ coming from the fact that $|S_k|$ specifies how many terms in the product~\eqref{eq:arb_intersection_Sk} are in the top branch (with a negative sign). If $S_k = [n]$ then it is a (non-partial) partition, and the partial partition has no singletons because $|S_\ell \setminus S_{\ell-1}| \geq 2$ for all $1 \leq \ell \leq k$.} This completes the proof of Theorem~\ref{thm:main_formula}.
\end{proof}

\subsection{The Polytope \texorpdfstring{$\mathbf{F_A}$}{FA}}\label{sec:polytope_flip_graph}

For an $n \times n$ matrix $A$ satisfying the conditions (a)--(c), we defined $F_A$ in Equation~\eqref{eq:FA_defn} by subtracting translated copies of the cuboid $C_A$ from the original cuboid $C_A$. In this subsection, we study the polytope $F_A$ and its 1-skeleton (i.e., its vertices and $1$-dimensional edges between them; see Figure~\ref{fig:1skeleton}), obtaining a description of it as a flip graph associated with a family of combinatorial objects.

To start, we describe the combinatorial objects that will be the vertices of the graph. An \emph{ordered partial partition} (\emph{OPP}) on $[n]$ is defined to be a relation $\preccurlyeq$ with the following properties:
\begin{itemize}
    \item[(i)] Transitive: if $x \preccurlyeq y$ and $y \preccurlyeq z$, then $x \preccurlyeq z$.
    \item[(ii)] Weakly reflexive: if $x \preccurlyeq y$, then $x \preccurlyeq x$ and $y \preccurlyeq y$.
    \item[(iii)] Weakly connected: if $x \preccurlyeq x$ and $y \preccurlyeq y$, then either $x \preccurlyeq y$ or $y \preccurlyeq x$ or both.
\end{itemize}
The set of all ordered partial partitions on $[n]$ is denoted by $\OPP(n)$ and the number of them is enumerated in \cite{oeisA000629}.

The name ``ordered partial partition'' refers to the fact that these relations on $[n]$ correspond bijectively with (ordered, possibly empty) tuples of pairwise disjoint non-empty subsets of $[n]$ (i.e., partial partitions of $[n]$ in which we care about the order of the parts):
\[
    (X_1, X_2, \ldots, X_k),
\]
where $X_i \cap X_j = \emptyset$ whenever $i \neq j$. This tuple of sets corresponds to the original relation in the following way: $x \preccurlyeq y$ if and only if there exist integers $1 \leq i \leq j$ such that $x \in X_i$ and $y \in X_j$.

For each $z \in [n]$, we define a function $f_z : \OPP(n) \rightarrow \OPP(n)$ as follows:
\begin{itemize}
    \item[(1)] If $X_1 = \{z\}$ then discard $X_1$, obtaining the tuple $(X_2, \ldots, X_k)$.
    
    \item[(2)] If $X_j = \{z\}$ for some $j \geq 2$, then merge it into the immediately preceding part, obtaining the tuple
    \[
        (X_1, \ldots, X_{j-1} \cup \{z\}, X_{j+1}, \ldots, X_k).
    \]
    
    \item[(3)] If $z \in X_j \neq \{z\}$ for some $j \geq 1$, then split it off into its own immediately next part, obtaining the tuple
    \[
        (X_1, \ldots, X_j \setminus \{z\}, \{z\}, X_{j+1}, \ldots, X_k).
    \]
    
    \item[(4)] Otherwise, if $z \notin \bigcup_{j=1}^k X_j$, then make it into its own first part, obtaining the tuple
    \[
        (\{z\}, X_{1}, \ldots, X_k).
    \]
\end{itemize}
For each $z \in [n]$, this function $f_z$ is an involution on $\OPP(n)$ with no fixed points: if $f_z(\preccurlyeq) = \preccurlyeq^\prime$, then $f_z(\preccurlyeq^\prime) = \preccurlyeq$. Moreover, if $x,y \in [n]$ with $x \neq z$, then $x \preccurlyeq y$ if and only if $x \preccurlyeq' y$.

We form a graph (called a \emph{flip graph}) by letting its vertices be the elements of $\OPP(n)$ and its edges be between ordered partial partitions that are exchanged by an involution $f_z$ for some $z \in [n]$. The resulting flip graph is $n$-regular, since each vertex gets one edge for each $z \in [n]$; we claim that this graph is isomorphic to the 1-skeleton of the polytope $F_A$ (see Figure~\ref{fig:1skeleton}).

\begin{figure}[!htb]
	\centering
	\begin{subfigure}[b]{.45\textwidth}
		\centering
		\begin{tikzpicture}[scale=0.76]%
      
            \draw[draw=green3!75!black] (0,0) --node[anchor=north]{$1$} (5,0) --node[anchor=west]{$2$} (5,2) --node[anchor=north]{$1$} (3,2) --node[anchor=east]{$2$} (3,3) --node[anchor=south]{$1$} (0,3) --node[anchor=east]{$2$} cycle;

        	\fill[green3!25!black] (0,0) circle (5pt) node[black,anchor=north east]{$()$};
        	\fill[green3!25!black] (5,0) circle (5pt) node[black,anchor=north west]{$(\{1\})$};
        	\fill[green3!25!black] (5,2) circle (5pt) node[black,anchor=south west]{$(\{2\},\{1\})$};
        	\fill[green3!25!black] (3,2) circle (5pt) node[black,anchor=north east]{$(\{1,2\})$};
        	\fill[green3!25!black] (3,3) circle (5pt) node[black,anchor=south west]{$(\{1\},\{2\})$};
        	\fill[green3!25!black] (0,3) circle (5pt) node[black,anchor=south east]{$(\{2\})$};
    	\end{tikzpicture}
		\caption{The $1$-skeleton of the 2D polytope $F_A$ from Figure~\ref{fig:our_tiling}, or equivalently the flip graph of $\OPP(2)$. Edges are labelled by the $z$ for which the involution $f_z$ transforms the OPPs that it connects to each other.}\label{fig:1skeleton_2d}
	\end{subfigure} \hfill 
	\begin{subfigure}[b]{.51\textwidth}
		\centering
        \tdplotsetmaincoords{56}{123}
    	\begin{tikzpicture}[scale=1,tdplot_main_coords]
    		\draw[draw=green3!75!black,dashed] (0,0,0) -- (3,0,0);
    		\draw[draw=green3!75!black,dashed] (0,0,0) -- (0,3,0);
    		\draw[draw=green3!75!black,dashed] (0,0,0) -- (0,0,3);
        	\fill[green3!75!black] (0,0,0) circle (3pt);
         
            \draw[draw=green3!75!black] (3,0,0) -- (3,0,2) -- (3,1,2) -- (3,1,1) -- (3,2,1) -- (3,2,0) -- cycle;
            \draw[draw=green3!75!black] (0,3,0) -- (0,3,2) -- (1,3,2) -- (1,3,1) -- (2,3,1) -- (2,3,0) -- cycle;
            \draw[draw=green3!75!black] (0,0,3) -- (0,2,3) -- (1,2,3) -- (1,1,3) -- (2,1,3) -- (2,0,3) -- cycle;
            \draw[draw=green3!75!black] (2,0,2) -- (2,0,3) -- (2,1,3) -- (2,1,2) -- cycle;
            \draw[draw=green3!75!black] (3,0,2) -- (2,0,2) -- (2,1,2) -- (3,1,2) -- cycle;
            \draw[draw=green3!75!black] (0,2,2) -- (0,2,3) -- (1,2,3) -- (1,2,2) -- cycle;
            \draw[draw=green3!75!black] (0,3,2) -- (0,2,2) -- (1,2,2) -- (1,3,2) -- cycle;
            \draw[draw=green3!75!black] (3,2,0) -- (2,2,0) -- (2,2,1) -- (3,2,1) -- cycle;
            \draw[draw=green3!75!black] (2,2,0) -- (2,3,0) -- (2,3,1) -- (2,2,1) -- cycle;
            \draw[draw=green3!75!black] (3,1,1) -- (3,1,2) -- (2,1,2) -- (2,1,3) -- (1,1,3) -- (1,1,1) -- cycle;
            \draw[draw=green3!75!black] (1,3,1) -- (1,3,2) -- (1,2,2) -- (1,2,3) -- (1,1,3) -- (1,1,1) -- cycle;
            \draw[draw=green3!75!black] (3,1,1) -- (3,2,1) -- (2,2,1) -- (2,3,1) -- (1,3,1) -- (1,1,1) -- cycle;
            
        	\fill[draw=green3!25!black] (3,0,0) circle (3pt) node[black,anchor=north east]{$(\{1\})$};
        	\fill[draw=green3!25!black] (0,3,0) circle (3pt) node[black,anchor=north west]{$(\{2\})$};
        	\fill[draw=green3!25!black] (0,0,3) circle (3pt) node[black,anchor=south]{$(\{3\})$};
        	\fill[draw=green3!25!black] (3,0,2) circle (3pt) node[black,anchor=east]{$(\{3\},\{1\})$};
        	\fill[draw=green3!25!black] (3,1,2) circle (3pt);
        	\fill[draw=green3!25!black] (3,1,1) circle (3pt);
        	\fill[draw=green3!25!black] (3,2,1) circle (3pt);
        	\fill[draw=green3!25!black] (3,2,0) circle (3pt);
        	\fill[draw=green3!25!black] (0,3,2) circle (3pt) node[black,anchor=west]{$(\{3\},\{2\})$};
        	\fill[draw=green3!25!black] (0,2,2) circle (3pt) node[black,anchor=south west]{$(\{2,3\})$};
        	\fill[draw=green3!25!black] (1,3,2) circle (3pt);
        	\fill[draw=green3!25!black] (1,3,1) circle (3pt);
        	\fill[draw=green3!25!black] (2,3,1) circle (3pt);
        	\fill[draw=green3!25!black] (2,3,0) circle (3pt);
        	\fill[draw=green3!25!black] (0,2,3) circle (3pt) node[black,anchor=south west]{$(\{2\},\{3\})$};
        	\fill[draw=green3!25!black] (1,2,3) circle (3pt);
        	\fill[draw=green3!25!black] (1,1,3) circle (3pt);
        	\fill[draw=green3!25!black] (2,1,3) circle (3pt);
        	\fill[draw=green3!25!black] (2,0,3) circle (3pt) node[black,anchor=south east]{$(\{1\},\{3\})$};
        	\fill[draw=green3!25!black] (1,1,1) circle (3pt);
        	\fill[draw=green3!25!black] (1,2,2) circle (3pt);
        	\fill[draw=green3!25!black] (2,2,0) circle (3pt);
        	\fill[draw=green3!25!black] (2,2,1) circle (3pt);
        	\fill[draw=green3!25!black] (2,0,2) circle (3pt) node[black,anchor=south east]{$(\{1,3\})$};
        	\fill[draw=green3!25!black] (2,1,2) circle (3pt);
    	\end{tikzpicture}
		\caption{The $1$-skeleton of the 3D polytope $F_B$ from Figure~\ref{fig:canon_matrix_3d}, or equivalently the flip graph of $\OPP(3)$. Not all vertices are labelled; the central vertex at the front is $(\{1,2,3\})$ while the central vertex at the back is $()$.}\label{fig:1skeleton_3d}
	\end{subfigure}
	
	\caption{The $1$-skeleton of a polytope is the set of its vertices and $1$-dimensional edges between them. Propositions~\ref{prop:FA_vertices} and~\ref{prop:FA_edges} show that the $1$-skeleton of the polytope $F_A$ is isomorphic to a graph whose vertices are the ordered partial partitions on $[n]$ and whose edges are described by the involutions $f_z$.}\label{fig:1skeleton}
\end{figure}
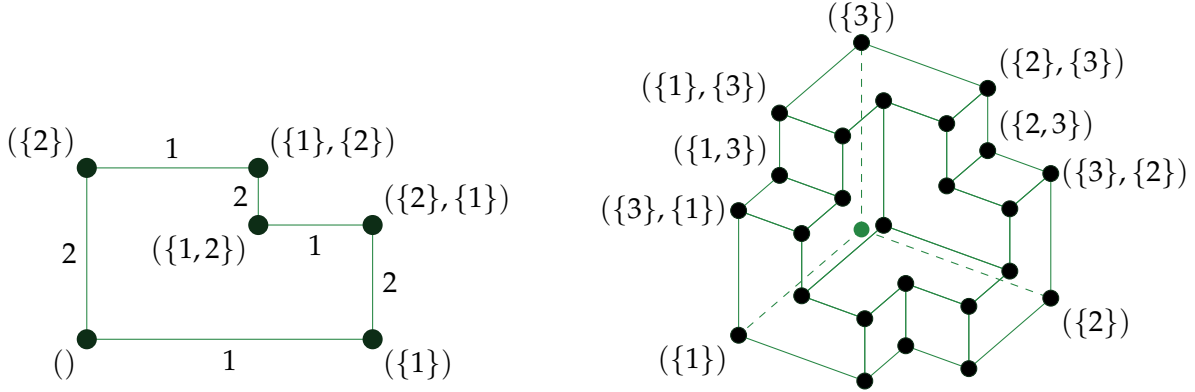

\begin{proposition}\label{prop:FA_vertices}
    Suppose that $A$ is an $n \times n$ matrix satisfying conditions (a)--(c). Then the vertices of $F_A$ are in bijective correspondence with the elements of $\OPP(n)$. In particular, the coordinates of the vertex $\mathbf{v}$ corresponding to the ordered partial partition $\preccurlyeq$ are given by
    \[
        v_i = \sum_{j : i \preccurlyeq j} a_{i,j} \quad \text{for all} \quad 1 \leq i \leq n.
    \]
\end{proposition}

\begin{proof}
    Recall that we originally defined $F_A$ by subtracting translated copies of the cuboid $C_A$:
    \begin{align}\label{eq:FA_defn_again}
        F_A := \overline{C_A \setminus \bigcup_{\emptyset \neq S \subseteq [n]}\left\{ x + \sum_{i \in S} A \mathbf{e_i} : x \in C_A \right\} }.
    \end{align}
    However, $C_A$ can itself be written as
    \[
        C_A = \overline{P \setminus \bigcup_{i \in [n]} \big\{ x + A \mathbf{e_i} : x \in P \big\}},
    \]
    where $P$ is the positive orthant, consisting of all vectors where every coordinate is non-negative. As such, the definition~\eqref{eq:FA_defn_again} of $F_A$ still works even if we replace each instance of $C_A$ with $P$.

    Given a vertex $\mathbf{v}$ of $F_A$, observe that it is the intersection of $n$ facets (codimension-1 faces) $f_1, f_2, \dots, f_n$ of $F_A$, where the standard basis vector $\mathbf{e_i}$ is perpendicular to the facet $f_i$. The facet $f_i$ must be contained in a facet of one of the translates of $P$ used to construct $F_A$; the vertex $\mathbf{v}$ is then the vertex of the intersection of these $n$ (not necessarily distinct) translates of the positive orthant $P$.

    Each vertex of $F_A$ that arises therefore corresponds to the vertex of an intersection of translated positive orthants, which is the lowest (in terms of each coordinate) vertex of the corresponding intersection of translated copies of $C_A$. We already characterised these intersections in the proof of Lemma~\ref{lem:FA_volume} when we performed the inclusion-exclusion calculation of the volume of $F_A$, identifying them with chains of subsets of $[n]$, which are in turn naturally identified with ordered partial partitions.\footnote{In the proof of Lemma~\ref{lem:FA_volume} we ignored the partial partitions containing singletons $\{i\}$, because they give measure-zero intersections. Here, however, they do still need to be included, because they correspond to vertices on a boundary face of the cuboid $C_A$ (specifically those vertices where $v_i = a_{i,i}$).}
\end{proof}

\begin{proposition}\label{prop:FA_edges}
    Let $A$ be an $n \times n$ matrix satisfying conditions (a)--(c) and let $1 \leq i \leq n$. Then two vertices $\mathbf{v}$ and $\mathbf{w}$ of $F_A$ are connected by an edge parallel to the standard basis vector $\mathbf{e_i}$ if and only if the corresponding (in the sense of Proposition~\ref{prop:FA_vertices}) elements of $\OPP(n)$ are swapped by the involution $f_i$.
\end{proposition}

\begin{proof}
    Combinatorially, the polytope $F_A$ does not depend on the choice of matrix $A$ (besides the fact that $A$ satisfies conditions (a)--(c)), so suppose without loss of generality that the entries are linearly independent over $\mathbb{Q}$.
    
    Then we have that $v_k = w_k$ if and only if $\sum_{j : k \preccurlyeq j} a_{k,j} = \sum_{j : k \preccurlyeq' j} a_{k,j}$ if and only if
    \begin{align}\label{eq:ordering_eq}
        \{j : k \preccurlyeq j\} = \{j : k \preccurlyeq' j\}.
    \end{align}
    Consequently, two vertices lie on a line parallel to the basis vector $\mathbf{e_i}$ if and only if Equation~\eqref{eq:ordering_eq} is true for all $k \neq i$.
    
    This implies that the ordered partial partitions induced on $[n] \setminus \{i\}$ are equal, so the original ordered partial partitions can only differ in terms of where $i$ is located. It also constrains the position of $i$ relative to the other parts, namely \begin{align}\label{eq:ordering_eq_i}
        \{k \neq i : k \preccurlyeq i\} = \{k \neq i : k \preccurlyeq' i\}.
    \end{align}
    
    If we remove and reinsert $i$, there are two possibilities: it is either introduced as its own singleton part, or it is not (either by being introduced into an existing part, or being deleted completely). In each of these two cases, the position of $i$ is determined by Equation~\eqref{eq:ordering_eq_i}. One of these cases corresponds to $\preccurlyeq^\prime = \preccurlyeq$, while the other corresponds to $\preccurlyeq^\prime = f_i(\preccurlyeq)$.
    
    Given that $F_A$ is an orthogonal polytope, it follows that there are at most $n$ other vertices that can possibly be connected to a given vertex $\mathbf{v}$ by an edge, namely the vertices obtained by applying $f_i$ to the corresponding element of $\OPP(n)$. Moreover, every vertex in an orthogonal polytope must have degree $n$, so all of these edges exist. The result follows.
\end{proof}

We can label the edges $\{\preccurlyeq, \preccurlyeq^\prime\}$ of the flip graph, as in Figure~\ref{fig:1skeleton_2d}, with the value of $i \in [n]$ for which $\preccurlyeq^\prime = f_i(\preccurlyeq)$. The higher-dimensional faces of $F_A$ then have concise descriptions in terms of this graph: the $r$-dimensional faces parallel to the linear subspace generated by the basis vectors $\{ \mathbf{e_i} : i \in S \}$ (where $S \subseteq [n]$ has $|S| = r$) are precisely the connected components of the subgraph obtained by taking only the edges whose labels are in $S$.

This description of $F_A$ by specifying its vertex coordinates is more general than the original definition in terms of subtracting translated copies of $C_A$: it generalises to arbitrary square matrices $A$, instead of requiring that $A$ satisfy the conditions (a)--(c). Note, however, that the resulting polytope may self-intersect and thus care is required to define ``volume'' in such a way that it is equal to $\det(A)$.

\section{Fields of Non-Zero Characteristic}\label{sec:fields_non_zero_char}

For fields $\F$ of characteristic $p > 0$, many of the terms in the formula~\eqref{eq:main_formula} are multiples of $p$ and thus equal to zero. As a result, the formula simplifies and we get an even tighter upper bound on the tensor rank of $\det_{\F}^n$. In particular, we have the following variant of Theorem~\ref{thm:main_formula}:

\begin{theorem}\label{thm:main_formula_nonzero}
    Let $A$ be an $n \times n$ matrix over a field $\F$ with characteristic $p > 0$. Then
    \begin{align}\label{eq:main_formula_nonzero}
        \det(A) & = \ \smashoperator{\sum_{\substack{P \in \PP(n) \\ |P| \leq p-1}}} \ \ \sgn(P) |P|! \prod_{i=1}^n \begin{cases}\displaystyle \ \ \sum_{j \underset{P}{\sim} i, j \neq i} a_{i,j} & \textrm{ if $i \underset{P}{\sim} i$;} \\ \displaystyle a_{i,i} + \sum_{j \underset{P}{\sim} j}a_{i,j} & \textrm{ if $i \underset{P}{\not\sim} i$,}\end{cases}
    \end{align}
    where $\displaystyle\sgn(P) = \prod_{S \in P}(-1)^{|S|+1}$.
\end{theorem}

\begin{proof}
    This formula is identical to the one provided by Theorem~\ref{thm:main_formula}, except the partial partitions $P$ that we sum over here are restricted to have $|P| \leq p-1$. Since any term with $|P| \geq p$ has a coefficient of $|P|!$, which is a multiple of $p$, which equals $0$ in $\F$, this restriction on $|P|$ does not change the value of the sum.
\end{proof}

In order to count the number of non-zero terms in the sum~\eqref{eq:main_formula_nonzero}, and thus obtain a tighter an upper bound on $\rank(\detnF)$ when $\F$ has characteristic $p > 0$, we need to introduce another combinatorial quantity. Let $B_{n,k}$ be the number of partial partitions of $[n]$ that contain exactly $k$ parts and no singletons, or equivalently the number of partitions of $[n]$ that contain exactly $k$ parts all of which have size strictly greater than $1$. The sum in Equation~\eqref{eq:main_formula_nonzero} contains exactly $\sum_{k=0}^{p-1}B_{n,k}$ (potentially) non-zero terms, so we immediately get the following corollary:

\begin{corollary}\label{cor:tensor_rank_field}
    Let $\F$ be a field with characteristic $p > 0$. Then $\displaystyle\rank(\detnF) \leq \sum_{k=0}^{p-1}B_{n,k}$.
\end{corollary}

By making use of Lemma~\ref{lem:TrankWrank_bound}, we could also say that if $\F$ has characteristic $p > n$ then $\Wrank(\detnF) \leq 2^{n-1}\sum_{k=0}^{p-1}B_{n,k}$. However, this provides no improvement over Corollary~\ref{cor:tensor_rank}, since $\sum_{k=0}^{p-1}B_{n,k} = B_n$ whenever $p > n$ (in fact, whenever $p > \lfloor n/2 \rfloor + 1$, since no partial partition can contain more than $\lfloor n/2 \rfloor$ parts unless it has some singletons).

Numerous properties and formulas for $B_{n,k}$ are known (see \cite{oeisA124324,NRB20} and the references therein). We summarize those that are most relevant for us here:

\begin{itemize}
    \item[(i)] $B_{n,0} = 1$.
    
    \item[(ii)] $B_{n,1} = \binom{n}{2} + \binom{n}{3} + \cdots + \binom{n}{n} = 2^n - n - 1$, since $\binom{n}{k}$ counts the number of partial partitions of $[n]$ with one part of size $k$. It follows that if $\F$ has characteristic $2$ then $\rank(\detnF) \leq B_{n,0} + B_{n,1} = 2^n - n$ (we return to this special case in more detail in Section~\ref{sec:permanent}).

    \item[(iii)] The recurrence relation $B_{n,k} = (k+1)B_{n-1,k} + (n-1)B_{n-2,k-1}$ holds whenever $n \geq 2k \geq 2$. When combined with the facts that $B_{n,0} = 1$ for all $n \geq 0$ and $B_{n,k} = 0$ whenever $n < 2k$, this recurrence relation can be used to compute $B_{n,k}$ for all $n$ and $k$.
    
    \item[(iv)] In particular, for small values of $k$ we have the following additional explicit formulas:
    \begin{align*}
        B_{n,2} & = \frac{1}{2}\big(3^n - (n+2)2^n + (n^2 + n + 1)\big), \\
        B_{n,3} & = \frac{1}{6}\big(4^n - (n+3)3^{n} + 3(n^2 + 3n + 4)2^{n-2} - (n^3 + 2n + 1)\big), \quad \text{and} \\
        B_{n,4} & = \frac{1}{24}\big(5^n - (n + 4)4^n + 2(n^2 + 5n + 9)3^{n-1} \\
        & \qquad \quad - (n^3 + 3n^2 + 8n + 8)2^{n-1} + (n^4 - 2n^3 + 5n^2 + 1)\big).
    \end{align*}
    
    \item[(v)] $B_{n,k} \leq (k+1)^n/k!$ (furthermore, for fixed $k$ we have $B_{n,k} \sim (k+1)^n/k!$, so this inequality is not too lossy). To verify this inequality, consider the following function $f$ from the set of $P \in \OPP(n)$\footnote{Recall that $\OPP(n)$ is the set of ordered partial partitions, defined in Section~\ref{sec:polytope_flip_graph}.} in which there are exactly $k$ parts and no singletons to the set $\{0,1,2,\ldots,k\}^n$:
    \begin{align*}
        f(P) := \; & (f_1(P), f_2(P), \ldots, f_n(P)), \quad \text{where} \\
        f_i(P) \; = \; & \begin{cases}
            0 & \text{if $i \underset{P}{\not\sim} i$}, \\
            j & \text{if $i$ is in the $j$-th part of $P$.}
        \end{cases}
    \end{align*}
    Since $f(P)$ completely specifies the ordered partial partition $P$, $f$ is injective, so there are no more than $|\{0,1,2,\ldots,k\}^n| = (k+1)^n$ different $P \in \OPP(n)$ with exactly $k$ parts and no singletons. If we then forget about the ordering of those $k$ parts, we see that there are no more than $(k+1)^n/k!$ members of $\PP(n)$ with exactly $k$ parts and no singletons, so $B_{n,k} \leq (k+1)^n/k!$.
\end{itemize}

In summary, all of the above observations and formulas, when combined with Corollary~\ref{cor:tensor_rank_field}, lead to the following (slightly weaker when $p \geq 7$, but much easier to evaluate and work with) corollary:

\begin{corollary}\label{cor:tensor_rank_field_simple}
    Let $\F$ be a field with characteristic $p > 0$.
    \begin{itemize}
        \item[(a)] If $p = 2$ then $\displaystyle \rank(\detnF) \leq 2^n - n$.
        
        \item[(b)] If $p = 3$ then $\displaystyle \rank(\detnF) \leq \frac{1}{2}\big(3^n - n2^n + (n^2 - n + 1)\big)$.
        
        \item[(c)] If $p = 5$ then $\displaystyle \rank(\detnF) \leq \frac{1}{24}\big(5^n - n4^n + 2(n^2 - n + 9)3^{n-1} - (n^3 - 3n^2 + 14n - 16)2^{n-1}$\\${}$\qquad\qquad\qquad\qquad\qquad\qquad\qquad $\displaystyle + (n^4 - 6n^3 + 17n^2 - 20n + 9)\big)$.
        
        \item[(d)] In general, $\displaystyle\rank(\detnF) \leq \sum_{k=1}^p\frac{k^n}{(k-1)!}$. In particular, $\displaystyle\rank(\detnF) \in O(p^n)$.
    \end{itemize}
\end{corollary}

\subsection{Rank of the Permanent Tensor}\label{sec:permanent}

Glynn's formula~\eqref{eq:perm_glynn} shows that the tensor rank of the permanent tensor is at most $2^{n-1}$, as long as the field does not have characteristic $2$. When the characteristic is $2$, the permanent and determinant tensors are the same, and the best known upper bound on their rank is now $2^n - n$, as described by Corollary~\ref{cor:tensor_rank_field_simple} (surpassing the $2^n - 1$ that comes from Ryser's formula~\eqref{eq:perm_ryser}). The following corollary clarifies slightly what this $2^n - n$ term formula for the determinant and permanent looks like in characteristic $2$. We note that the ``$\ominus$'' in the following corollary is the symmetric difference operation on sets, and $\pernF$ refers to the determinant tensor $\pernF = \sum_{\sigma \in S_n} \mathbf{e}_{\sigma(1)} \otimes \mathbf{e}_{\sigma(2)} \otimes \cdots \otimes \mathbf{e}_{\sigma(n)}$:

\begin{corollary}\label{cor:tensor_rank_permanent_char2}
    Let $A$ be an $n \times n$ matrix over a field $\F$ with characteristic $2$. Then
    \begin{align}\label{eq:main_formula_char2}
        \per(A) = \det(A) & = \sum_{S \subseteq [n]} \left( \prod_{i=1}^n \sum_{j \in S\ominus\{i\}} a_{i,j} \right).
    \end{align}
    In particular, whenever $S$ is a singleton the inner sum is empty, so $\rank(\pernF) = \rank(\detnF) \leq 2^n - n$.
\end{corollary}

The above corollary comes directly from plugging $p = 2$ into Theorem~\ref{thm:main_formula_nonzero} and discarding all terms coming from partial partitions that have two or more parts. The partial partitions with just zero or one part are in bijection with the subsets $S$ of $[n]$, giving the result.

For $n \leq 3$, the formula provided by Corollary~\ref{cor:tensor_rank_permanent_char2} is the same as the field-independent formula that we have seen already (except with signs ignored since $-1 = 1$ in characteristic $2$). The first non-trivial case comes when $n = 4$. In this case, the resulting formula has $12$ terms; it is the same as the formula displayed in Equation~\eqref{eq:det_n4_formula}, but without the final $3$ terms (i.e., the terms with a coefficient of $2$).

\section{Tensor Rank Lower bounds}\label{sec:lower_bound}

All of the results that we have presented so far bound $\rank(\detnF)$ and $\rank(\pernF)$ from above. In the other direction, the best known lower bounds on $\rank(\detnF)$ and $\rank(\pernF)$ when $n \geq 4$ are due to Derksen~\cite{Der15}, who showed that
\begin{align}\label{eq:derksen_lb}
    \rank(\detnF) \geq \binom{n}{\lfloor n/2 \rfloor} \quad \text{and} \quad \rank(\pernF) \geq \binom{n}{\lfloor n/2 \rfloor} 
\end{align}
(he only stated these bounds for $\F = \C$, but his proof works over any field).\footnote{There are some better lower bounds when $n \in \{5,7\}$ in \cite{KM21}.}

The known upper bounds show that this lower bound on the permanent cannot be improved by more than a factor of $\Theta(\sqrt{n})$ over any field. In particular, if $\F$ has characteristic not equal to $2$ then Glynn's formula~\eqref{eq:perm_glynn} implies
\[
    2^n\sqrt{\frac{2}{\pi n}} \thicksim \binom{n}{\lfloor n/2 \rfloor} \leq \rank(\pernF) \leq 2^{n-1}.
\]
Similarly, Corollary~\ref{cor:tensor_rank_permanent_char2} tells us that if $\F$ has characteristic $2$ then  we have
\[
    2^n\sqrt{\frac{2}{\pi n}} \thicksim \binom{n}{\lfloor n/2 \rfloor} \leq \rank(\detnF) = \rank(\pernF) \leq 2^n - n.
\]
We can, however, make some small improvements. In particular, we refine Derksen's analysis slightly to improve these lower bounds by $1$:

\begin{theorem}\label{thm:det_rank_lb}
    Let $\mathbb{F}$ be any field. Then $\displaystyle\rank(\detnF) \geq \binom{n}{\lfloor n/2\rfloor} + 1$ and $\displaystyle\rank(\pernF) \geq \binom{n}{\lfloor n/2\rfloor} + 1$.
\end{theorem}

\begin{proof}
    We begin by recalling the definition of the antisymmetric subspace $\mathcal{A}$ of $(\mathbb{F}^n)^{\otimes s}$:\footnote{\label{foot:f2}There are other definitions of the antisymmetric subspace (see \cite[Section~3.1.3]{JohALA}, for example) that are equivalent over fields of characteristic not equal to $2$, but in characteristic $2$ it is important that we choose this definition in this proof.}
    \begin{align}\label{eq:antisym_subspace}
        \mathcal{A} \defeq \mathrm{span}\left\{ \sum_{\sigma \in S_s} \sgn(\sigma) \mathbf{e}_{i_{\sigma(1)}} \otimes \mathbf{e}_{i_{\sigma(2)}} \otimes \cdots \otimes \mathbf{e}_{i_{\sigma(s)}} : 1 \leq i_1 < i_2 < \cdots < i_s \leq n \right\}.
    \end{align}
    
    We just prove the inequality in the statement of the theorem that involves $\rank(\detnF)$; the bound on $\rank(\pernF)$ follows similarly by just ignoring signs throughout the proof and omitting $\sgn(\sigma)$ from the definition~\eqref{eq:antisym_subspace} of the `antisymmetric subspace' $\mathcal{A}$.\footnote{For example, if $n = s = 2$ then this means that $\mathcal{A} = \mathrm{span}\{\mathbf{e_1} \otimes \mathbf{e_2} + \mathbf{e_2} \otimes \mathbf{e_1}\}$; it does \emph{not} mean that $\mathcal{A}$ is the symmetric subspace (it does not contain $\mathbf{e_1} \otimes \mathbf{e_1}$ or $\mathbf{e_2} \otimes \mathbf{e_2}$).}
    
    It is well-known that if $T$ is any matrix flattening of $\detnF$ then $\rank(\detnF) \geq \operatorname{rank}(T)$. We choose $T$ to be the flattening obtained by partitioning $(\mathbb{F}^n)^{\otimes n}$ as $(\mathbb{F}^n)^{\otimes \lfloor n/2\rfloor} \otimes (\mathbb{F}^n)^{\otimes \lceil n/2\rceil}$, which we think of as an $n^{\lfloor n/2\rfloor} \times n^{\lceil n/2\rceil}$ matrix. The columns of this matrix $T$ span the antisymmetric subspace $\mathcal{A}$ of $(\mathbb{F}^n)^{\otimes \lfloor n/2\rfloor}$, which has dimension $\binom{n}{\lfloor n/2 \rfloor}$, so
    \begin{align}\label{eq:in_proof_derksen_bound}
        \rank(\detnF) \geq \operatorname{rank}(T) = \binom{n}{\lfloor n/2 \rfloor},
    \end{align}
    recovering Derksen's lower bound.
    
    To show that Inequality~\eqref{eq:in_proof_derksen_bound} is actually strict (and thus complete the proof), note that for any tensor decomposition of the form
    \begin{align*}
        \detnF = \sum_{k=1}^r \mathbf{v_{1,k}} \otimes \mathbf{v_{2,k}} \otimes \cdots \otimes \mathbf{v_{n,k}},
    \end{align*}
    we have
    \begin{align}\label{eq:det_matrix_flattening}
        T = \sum_{k=1}^r \big(\mathbf{v_{1,k}} \otimes \cdots \otimes \mathbf{v_{\lfloor n/2\rfloor,k}}\big)\big(\mathbf{v_{\lfloor n/2\rfloor+1,k}} \otimes \cdots \otimes \mathbf{v_{n,k}}\big)^T.
    \end{align}
    
    In particular, this implies that the range of $T$ (i.e., the antisymmetric subspace of $(\mathbb{F}^n)^{\otimes \lfloor n/2\rfloor}$) is contained in the span of the $r$ elementary tensors $\mathbf{v_{1,k}} \otimes \cdots \otimes \mathbf{v_{\lfloor n/2\rfloor,k}}$ ($1 \leq k \leq r$). If $r \leq \binom{n}{\lfloor n/2 \rfloor}$ (the dimension of the antisymmetric subspace) then this implies that each of these elementary tensors are in the antisymmetric subspace, contradicting the fact that \emph{no} non-zero elementary tensors are in the antisymmetric subspace.\footnote{In the case of the permanent, we can see that $\mathcal{A}$ contains no non-zero elementary tensors $\mathbf{v} \otimes \cdots \otimes \mathbf{v}$ by noting that such a tensor has a non-zero ``diagonal'' entry (i.e., there exists some $i$ for which $(\mathbf{e_i} \otimes \cdots \otimes \mathbf{e_i})^*(\mathbf{v} \otimes \cdots \otimes \mathbf{v}) \neq 0$), but every member of the antisymmetric subspace $\mathcal{A}$ has all diagonal entries equal to $0$.} We thus conclude that $r > \binom{n}{\lfloor n/2 \rfloor}$, which completes the proof.
\end{proof}

\subsection{Better Lower Bounds Over Finite Fields}\label{sec:lower_bound_f2}

In the case when $\F = \F_q$ is the field with $q$ elements, we can refine the argument of Theorem~\ref{thm:det_rank_lb} even further to get even better lower bounds on $\rank(\det_{\F_q}^n)$. In particular, we have the following:

\begin{theorem}\label{thm:det_rank_lb_f2}
    Let $q \geq 2$ be a prime power, let $n \geq 5$, and let $x$ be the (unique) positive real solution to the equation $\log_q(x + 1) = x - \binom{n}{\lfloor n/2 \rfloor}$. Then
    \[
        \rank(\det_{\F_q}^n) \geq \lceil x \rceil.
    \]
    In particular, $\rank(\det_{\F_q}^n) \geq \binom{n}{\lfloor n/2 \rfloor} + \log_q\big(\binom{n}{\lfloor n/2 \rfloor}\big)$.
\end{theorem}

Before we can prove this, we begin by establishing a useful lemma about the tensor rank of members of the antisymmetric subspace:\footnote{See footnote~\ref{foot:f2}.}

\begin{lemma}\label{lem:antisymmetric_rank}
Suppose that $\mathbb{F}$ is a field, $n, s$ are positive integers, and $T$ is a nonzero element of the antisymmetric subspace of $(\mathbb{F}^n)^{\otimes s}$. Then $\rank(T) \geq s$.
\end{lemma}

\begin{proof}
    If $s = 1$, the result is immediate, because every nonzero tensor has positive tensor rank. We shall henceforth assume that $s \geq 2$.
    
    We view the tensor $T$ as a multilinear form from $(\mathbb{F}^n)^{s}$ to $\mathbb{F}$. By assumption that it is nonzero, there must exist vectors $\mathbf{x_1}, \ldots, \mathbf{x_s} \in \mathbb{F}^n$ such that $T(\mathbf{x_1}, \ldots, \mathbf{x_s})$ is nonzero. By multiplying one of these vectors by an appropriate scalar, we can assume without loss of generality that $T(\mathbf{x_1}, \dots, \mathbf{x_s}) = 1$.
    
    Note that $\{\mathbf{x_1}, \dots, \mathbf{x_s}\}$ must be linearly independent. Otherwise, we could write one of the vectors as a linear combination of the others, which we will assume without loss of generality is $\mathbf{x_1}$. That is, $\mathbf{x_1} = \alpha_2 \mathbf{x_2} + \cdots + \alpha_s \mathbf{x_s}$, and then we would have the following by linearity in the first argument:
    \[
        T(\mathbf{x_1}, \mathbf{x_2}, \ldots, \mathbf{x_s}) = \sum_{i=2}^s \alpha_i T(\mathbf{x_i}, \mathbf{x_2}, \ldots, \mathbf{x_s}).
    \]
    By antisymmetry, each term on the right-hand side vanishes, and the left-hand side equals 1, so we obtain a contradiction. As such, it follows that $\{\mathbf{x_1}, \dots, \mathbf{x_s}\}$ is indeed linearly independent.

    For a fixed $1 \leq i \leq s$, we define a covector $\mathbf{y_i} : \mathbb{F}^n \rightarrow \mathbb{F}$ by contracting $T$ on its first $s - 1$ arguments with the elements of $\{\mathbf{x_1}, \dots, \mathbf{x_s}\} \setminus \{ \mathbf{x_i} \}$ and applying an appropriate sign change:
    
    \[ \mathbf{y_i}(\mathbf{v}) := (-1)^{s - i} T(\mathbf{x_1}, \dots, \mathbf{x_{i-1}}, \mathbf{x_{i+1}}, \dots, \mathbf{x_s}, \mathbf{v}) \]
    
    Observe that the antisymmetry properties of $T$ imply the following:

    \[
        \mathbf{y_i}(\mathbf{x_j}) = \begin{cases}
        1 & \textrm{ if $i = j$;} \\
        0 & \textrm{ if $i \neq j$.}
        \end{cases}
    \]
    Equivalently, if we consider the space $V$ spanned by $\mathbf{x_1}, \ldots, \mathbf{x_s}$, then the covectors $\mathbf{y_1}, \ldots, \mathbf{y_s}$ form a basis for $V^{\star}$, specifically the dual basis of $\mathbf{x_1}, \ldots, \mathbf{x_s}$.
    
    Consequently, the matrix obtained from $T$ by the flattening $(\mathbb{F}^n)^{\otimes s} = (\mathbb{F}^n)^{\otimes (s-1)} \otimes (\mathbb{F}^n)$ has tensor rank at least $s$, and thus $\rank(T) \geq s$.
\end{proof}

\begin{proof}[Proof of Theorem~\ref{thm:det_rank_lb_f2}]
    We showed in the proof of Theorem~\ref{thm:det_rank_lb} that if $1 \leq s \leq n$ and
    \[
        \det_{\F_q}^n = \sum_{k=1}^r \mathbf{v_{1,k}} \otimes \mathbf{v_{2,k}} \otimes \cdots \otimes \mathbf{v_{n,k}}
    \]
    then the antisymmetric subspace of $(\mathbb{F}_q^n)^{\otimes s}$ must be contained in the span of the $r$ elementary tensors $\mathbf{v_{1,k}} \otimes \cdots \otimes \mathbf{v_{s,k}}$ ($1 \leq k \leq r$). We will use this fact with $s = \lceil n/2 \rceil$ (in the proof of Theorem~\ref{thm:det_rank_lb} we instead used $s = \lfloor n/2\rfloor$).

    Let $P : (\F_q^n)^{\otimes \lceil n/2\rceil} \rightarrow (\F_q^n)^{\otimes \lceil n/2\rceil} / \mathcal{A}$ be a projection onto the quotient of the antisymmetric subspace $\mathcal{A}$ inside of $(\F_q^n)^{\otimes \lceil n/2\rceil}$. If the antisymmetric subspace of $(\mathbb{F}_q^n)^{\otimes \lceil n/2\rceil}$ is contained in the span of $r$ elementary tensors $\mathbf{v_{1,j}} \otimes \cdots \otimes \mathbf{v_{\lceil n/2\rceil,j}}$ ($1 \leq j \leq r$), then the span of the set
    \begin{align}\label{eq:F2_span}
        B := \big\{P(\mathbf{v_{1,j}} \otimes \cdots \otimes \mathbf{v_{\lceil n/2\rceil,j}}) : 1 \leq j \leq r \big\}
    \end{align}
    must have dimension at most $r - \binom{n}{\lfloor n/2 \rfloor}$. Since we are working over $\F_q$, there are thus at most
    \[
        q^{r - \binom{n}{\lfloor n/2 \rfloor}} - 1
    \]
    non-zero vectors in~$B$.
    
    Next, we note that there are exactly $r$ members in $B$ (i.e., the vectors $P(\mathbf{v_{1,i}} \otimes \cdots \otimes \mathbf{v_{\lceil n/2\rceil,i}})$ and $P(\mathbf{v_{1,j}} \otimes \cdots \otimes \mathbf{v_{\lceil n/2\rceil,j}})$ are distinct whenever $i \neq j$). To see why this is the case, we apply Lemma~\ref{lem:antisymmetric_rank}: given any tensor $T$ in the antisymmetric subspace of $(\mathbb{F}_q^n)^{\otimes \lceil n/2\rceil}$, we have $\rank(T) \geq \lceil n/2 \rceil \geq 3$, which implies that $P(\mathbf{v_{1,i}} \otimes \cdots \otimes \mathbf{v_{\lceil n/2\rceil,i}} - \mathbf{v_{1,j}} \otimes \cdots \otimes \mathbf{v_{\lceil n/2\rceil,j}}) = \mathbf{0}$ never occurs.

    It follows that the set~$B$ consists of exactly $r$ non-zero vectors, but also no more than $q^{r - \binom{n}{\lfloor n/2 \rfloor}} - 1$ non-zero vectors, so we must have $r \leq q^{r - \binom{n}{\lfloor n/2 \rfloor}} - 1$. Rearranging shows that $\log_q(r + 1) \leq r - \binom{n}{\lfloor n/2 \rfloor}$. Since the function
    \[
        f(x) := \log_q(x + 1) - x + \binom{n}{\lfloor n/2 \rfloor}
    \]
    is monotonically decreasing on $(0,\infty)$, has $f(1) > 0$, and $\lim_{x\rightarrow\infty}f(x) = -\infty$, we conclude that there is exactly one positive $\widetilde{x} \in \R$ for which $f(\widetilde{x}) = 0$. The least integer $r \geq 1$ for which $\log_q(r + 1) \leq r - \binom{n}{\lfloor n/2 \rfloor}$ is $r = \lceil \widetilde{x} \rceil$. The ``in particular'' statement of the theorem comes from the fact that $f(x) > 0$ when $x = \binom{n}{\lfloor n/2 \rfloor} + \log_q\big(\binom{n}{\lfloor n/2 \rfloor}\big)$ too.
\end{proof}

When $n = 4$, if we apply the same reasoning as in the proof of Theorem~\ref{thm:det_rank_lb_f2} then we have to be slightly careful, since in this case $\lceil n/2\rceil = 2$ so there are antisymmetric tensors in $(\F_q^n)^{\otimes \lceil n/2\rceil}$ with rank $2$, leading to fewer than $r$ members of the set~$B$ from Equation~\eqref{eq:F2_span}. Instead of the inequality $r \leq q^{r - \binom{n}{\lfloor n/2 \rfloor}} - 1$, we obtain the weaker inequality $r/2 \leq q^{r - \binom{n}{\lfloor n/2 \rfloor}} - 1$, which leads to the bound
\begin{align}\label{eq:D4_lb}
    \rank(\det_{\F_q}^4) \geq \begin{cases}9 & \textrm{ if $q = 2$;} \\ 8 & \textrm{ if $q \in \{3, 4\}$;} \\ 7 & \textrm{ if $q \geq 5$ is any other prime power.}\end{cases}
\end{align}

We can actually prove a much stronger lower bound in the $n = 4$, $q = 2$ case, which matches our upper bound from Theorem~\ref{thm:main_formula_nonzero}. In particular, we have the following theorem, whose proof is computer-assisted and described in the next subsection:

\begin{theorem}\label{thm:exactly_12}
    Over the field of two elements, the tensor rank of the determinant and permanent of a $4 \times 4$ matrix is exactly 12:
    \[
        \rank(\det_{\F_2}^4) = \rank(\per_{\F_2}^4) = 12.
    \]
\end{theorem}

\subsection{Proof That the Tensor Rank of \texorpdfstring{$\det_4^{\F_2}$}{det4 over F2} is 12}\label{sec:twelve_proof}

The formula in Theorem~\ref{thm:main_formula_nonzero} shows that the tensor rank is at most 12, and we have already shown a lower bound of 9. It suffices, therefore, to show that it is impossible to express the determinant tensor as the sum of $r \in \{ 9, 10, 11 \}$ rank-1 tensors.

Suppose otherwise. By considering the flattening $(\mathbb{F}_2^4)^{\otimes 4} = (\mathbb{F}_2^4)^{\otimes 2} \otimes (\mathbb{F}_2^4)^{\otimes 2}$, we can write
\[
    \det_4^{\F_2} = \sum_{i=1}^r A_i \otimes B_i
\]
where each $A_i$ and $B_i$ is a $4 \times 4$ rank-1 matrix. Recall that the span of $\{ A_i : 1 \leq i \leq r \}$ must contain the (6-dimensional) antisymmetric subspace of $(\mathbb{F}_2^4)^{\otimes 2}$.

Without loss of generality, we can assume that $A_1 \leq A_2 \leq \cdots \leq A_r$, where $\leq$ denotes the lexicographical order on the space of $4 \times 4$ matrices over $\F_2$. Moreover, given any such tensor decomposition of $\det_4^{\F_2}$ and an invertible matrix $L$, observe that the determinant is unaffected by the change of basis specified by $L$, so we also have
\[
    \det_4^{\F_2} = \sum_{i=1}^r L A_i L^T \otimes L B_i L^T.
\]

By applying a suitable change of basis $L$, we can without loss of generality assume that at least one of the matrices---necessarily the lexicographically first matrix, and therefore $A_1$ by definition---has a single $1$ in either the last or penultimate entry of the matrix, and zeroes in all other entries:

\[ A_1 = \begin{bmatrix}
0 & 0 & 0 & 0 \\
0 & 0 & 0 & 0 \\
0 & 0 & 0 & 0 \\
0 & 0 & 0 & 1 \end{bmatrix} \textrm{ or } \begin{bmatrix}
0 & 0 & 0 & 0 \\
0 & 0 & 0 & 0 \\
0 & 0 & 0 & 0 \\
0 & 0 & 1 & 0 \end{bmatrix} \]

We performed a backtracking search\footnote{Source code is available at: \url{https://gitlab.com/apgoucher/det4f2}} for candidate $r$-tuples $(A_1, \dots, A_r)$ of rank-1 matrices subject to these constraints. If the intersection of the 6-dimensional antisymmetric subspace and the span of the initial segment $(A_1, \dots, A_s)$ has rank less than $6 - r + s$, then it is impossible for the span of $(A_1, \dots, A_r)$ to contain the antisymmetric subspace, and therefore we can eliminate that entire branch of the search tree.

The largest search ($r = 11$) consumed 357 core-hours of CPU time, or 4 hours of wall-clock time when parallelised on an AWS r5a.24xlarge instance. The total CPU time grows rapidly as a function of $r$:

\begin{center}
\begin{tabular}{c | c}
$r$ & CPU time \\\hline
9 & 7 seconds \\
10 & 22 minutes \\
11 & 357 hours \\
\end{tabular}
\end{center}

It transpires that, when $r \geq 9$, there do exist $r$-tuples of rank-1 matrices $A_1, \dots, A_r$ which contain the antisymmetric subspace in their linear span, so additional ideas are necessary to eliminate these candidates.

\begin{lemma}
    Suppose that $\det_4^{\F_2} = \sum_{i=1}^r A_i \otimes B_i$ and $r$ is minimal. Let $\mathbf{u}, \mathbf{v} \in \F_2^4$ be a pair of (not necessarily distinct) nonzero vectors. Then the size of the set of indices $\{ i : \mathbf{u}^T A_i \mathbf{v} = 1 \}$ is not equal to 1.
\end{lemma}

\begin{proof}
    Suppose otherwise, namely that there is a unique index $i$ for which $\mathbf{u}^T A_i \mathbf{v} = 1$.
    
    If we contract both sides of the equation $\det_4^{\F_2} = \sum_{i=1}^r A_i \otimes B_i$ with $\mathbf{u}$ and $\mathbf{v}$ on the first tensor factor, all but the $i$-th term of the sum will vanish and we obtain (in index notation)
    \[
        u^j v^k \det_4^{j,k,\ell,m} = (B_i)^{\ell,m}.
    \]
    
    If $\mathbf{u} = \mathbf{v}$, then the left-hand-side is the all-zeroes matrix, and therefore so is $B_i$. But that means that we have a tensor decomposition of rank $r - 1$, contradicting minimality. Otherwise, $\mathbf{u} \neq \mathbf{v}$ and we can apply a suitable change of basis so that, without loss of generality, $\mathbf{u} = \mathbf{e_1}$ and $\mathbf{v} = \mathbf{e_2}$. Then the left-hand-side of the equation is the rank-2 matrix $\mathbf{e_3} \otimes \mathbf{e_4} + \mathbf{e_4} \otimes \mathbf{e_3}$, but the right-hand-side has rank~1, again obtaining a contradiction.
\end{proof}

This lemma eliminates all candidate 9-tuples and 10-tuples, establishing a tensor rank lower bound of 11, and leaves a single candidate 11-tuple up to change of basis and up to transposing the matrices $A_i$:

\begin{equation} \begin{gathered} A_1 = \begin{bmatrix}
0 & 0 & 0 & 0 \\
0 & 0 & 0 & 0 \\
0 & 0 & 0 & 0 \\
0 & 0 & 1 & 0 \end{bmatrix}, A_2 = \begin{bmatrix}
0 & 0 & 0 & 0 \\
0 & 0 & 0 & 0 \\
0 & 0 & 0 & 1 \\
0 & 0 & 0 & 0 \end{bmatrix}, A_3 = \begin{bmatrix}
0 & 0 & 0 & 0 \\
0 & 1 & 0 & 0 \\
0 & 1 & 0 & 0 \\
0 & 1 & 0 & 0 \end{bmatrix}, A_4 = \begin{bmatrix}
0 & 0 & 0 & 0 \\
0 & 1 & 0 & 1 \\
0 & 1 & 0 & 1 \\
0 & 0 & 0 & 0 \end{bmatrix}, \\
\textcolor{green3!75!black}{A_5 = \begin{bmatrix}
0 & 0 & 0 & 0 \\
0 & 1 & \textbf{1} & 0 \\
0 & 0 & 0 & 0 \\
0 & 1 & 1 & 0 \end{bmatrix}}, A_6 = \begin{bmatrix}
1 & 0 & 0 & 0 \\
0 & 0 & 0 & 0 \\
1 & 0 & 0 & 0 \\
1 & 0 & 0 & 0 \end{bmatrix}, A_7 = \begin{bmatrix}
1 & 0 & 0 & 1 \\
0 & 0 & 0 & 0 \\
1 & 0 & 0 & 1 \\
0 & 0 & 0 & 0 \end{bmatrix}, \textcolor{green3!75!black}{A_8 = \begin{bmatrix}
1 & 0 & \textbf{1} & 0 \\
0 & 0 & 0 & 0 \\
0 & 0 & 0 & 0 \\
1 & 0 & 1 & 0 \end{bmatrix}}, \\
A_9 = \begin{bmatrix}
1 & 1 & 0 & 0 \\
1 & 1 & 0 & 0 \\
1 & 1 & 0 & 0 \\
1 & 1 & 0 & 0 \end{bmatrix}, A_{10} = \begin{bmatrix}
1 & 1 & 0 & 1 \\
1 & 1 & 0 & 1 \\
1 & 1 & 0 & 1 \\
0 & 0 & 0 & 0 \end{bmatrix}, \quad \text{and} \quad \textcolor{green3!75!black}{A_{11} = \begin{bmatrix}
1 & 1 & \textbf{1} & 0 \\
1 & 1 & \textbf{1} & 0 \\
0 & 0 & 0 & 0 \\
1 & 1 & 1 & 0 \end{bmatrix}}. \end{gathered} \end{equation}

The only matrices in this set to have a nonzero $(1,3)$ entry are $A_8$ and $A_{11}$; the only matrices to have a nonzero $(2,3)$ entry are $A_5$ and $A_{11}$. By contracting both sides of the equation $\det_4^{\F_2} = \sum_{i=1}^r A_i \otimes B_i$ with $\mathbf{e_1} \otimes \mathbf{e_3}$, we get $B_8 + B_{11} = \mathbf{e_2} \otimes \mathbf{e_4} + \mathbf{e_4} \otimes \mathbf{e_2}$. Similarly, by contracting both sides of that equation by $\mathbf{e_2} \otimes \mathbf{e_3}$, we get $B_5 + B_{11} = \mathbf{e_1} \otimes \mathbf{e_4} + \mathbf{e_4} \otimes \mathbf{e_1}$.

There are only 6 different ways to write each of these rank-2 matrices as the sum of two rank-1 matrices. In particular, we must have
\[
    B_8, B_{11} \in \{ \mathbf{e_2} \otimes \mathbf{e_4}, \mathbf{e_4} \otimes \mathbf{e_2}, \mathbf{e_2} \otimes (\mathbf{e_2} + \mathbf{e_4}), (\mathbf{e_2} + \mathbf{e_4}) \otimes \mathbf{e_2}, (\mathbf{e_2} + \mathbf{e_4}) \otimes \mathbf{e_4}, \mathbf{e_4} \otimes (\mathbf{e_2} + \mathbf{e_4}) \}
\]
and
\[
    B_5, B_{11} \in \{ \mathbf{e_1} \otimes \mathbf{e_4}, \mathbf{e_4} \otimes \mathbf{e_1}, \mathbf{e_1} \otimes (\mathbf{e_1} + \mathbf{e_4}), (\mathbf{e_1} + \mathbf{e_4}) \otimes \mathbf{e_1}, (\mathbf{e_1} + \mathbf{e_4}) \otimes \mathbf{e_4}, \mathbf{e_4} \otimes (\mathbf{e_1} + \mathbf{e_4}) \}.
\]
Observe that these two sets are disjoint, so $B_{11}$ cannot be in both of them. This is a contradiction that eliminates this single candidate solution for $r = 11$, completing the proof of Theorem~\ref{thm:exactly_12}.\bigskip

\noindent \textbf{Acknowledgements.} Thanks to Dan Piker for many Twitter conversations and excellent graphics, one of which set the first author on the road that led to this paper. Thanks to Zach Teitler for clarifying some facts about Waring rank that appeared in Section~\ref{sec:waring} \cite{Tei22}. N.J.\ was supported by NSERC Discovery Grant RGPIN-2022-04098. Thanks also to an anonymous reviewer for numerous helpful corrections and suggestions.

\bibliographystyle{alpha}
\bibliography{bib}

\end{document}